%% file: acceleration.tex
\documentclass{siamart}

\usepackage{tikz}
\usepackage{pgfplots}
\usetikzlibrary{plotmarks,patterns}
	\newlength\figureheight 
	\newlength\figurewidth 
 \pgfplotsset{
     tick label style={font=\footnotesize},
     label style={font=\footnotesize},
     legend style={font=\scriptsize}
 }

\usepackage{color}

\definecolor{blockbg}{RGB}{195,195,195}
\definecolor{tblockbg}{RGB}{0,140,195}
\definecolor{eblockbg}{RGB}{60,179,113}
\definecolor{ablockbg}{RGB}{240,90,90}

\usepackage{graphicx}  
\usepackage{psfrag}    
\usepackage{url}       
\usepackage{amssymb,amsmath,amscd,amsfonts,amstext,bbm, enumerate, ntheorem}
\usepackage{array}
\usepackage{multirow}
\usepackage[latin1]{inputenc}
\usepackage{eclbkbox}
\usepackage{multicol}

\usepackage{caption}
\usepackage{subcaption}
\usepackage{graphicx}

\newcommand{\op}[1]{\ifmmode\mathsf{#1}\else\textsf{#1}\fi}
\newcommand{\mat}[1]{\ifmmode\mathbf{#1}\else\textbf{#1}\fi}

\DeclareMathOperator*{\argmin}{argmin}

\newcommand\tran{\mathrm{T}}
\newcommand\fix{\mathrm{fix}}

\newtheorem{remark}{Remark}
\newtheorem{example}{Example}

\begin{document}

\title{A Generic online acceleration scheme for Optimization algorithms via Relaxation and Inertia}
\author{F. Iutzeler and J. M. Hendrickx \thanks{F.I. is with LJK, Universit\'{e} Grenoble Alpes, Grenoble, France. J.H. is with ICTEAM, Universit\'{e} Catholique de Louvain, Louvain-la-Neuve, Belgium. This project was conducted while F.I. was a post-doctoral researcher at UCL and is supported by the Belgian Network DYSCO, funded by the Belgian government and the Concerted Research Action (ARC) of the French Community of Belgium.} }

\date{\today}

\markboth{Draft}{}

\maketitle

\begin{abstract}
We propose generic acceleration schemes for a wide class of optimization and iterative schemes based on relaxation and inertia. In particular, we introduce methods that automatically tunes the acceleration coefficients online, and establish their convergence. This is made possible by considering the class of fixed-points iterations over averaged operators which encompass gradient methods, ADMM, primal dual algorithms, an so on. 



\end{abstract}

\begin{keywords}
Applied Optimization Methods, Relaxation, Inertia, Acceleration.
\end{keywords} 



\section{Introduction} 
\label{sec:intro}

A large class of optimization algorithms can be cast as {fixed-point} iterations in the sense that they consist in applying the same operation successively in order to converge to a fixed point of this operation. For the gradient algorithm on a differentiable function $f$, the operation consists in applying the identity minus the gradient of $f$, and the fixed point reached nulls the gradient of $f$. The convergence of such fixed-points iterations can be proven by finding a suitable contraction property, for which the \emph{monotone operators} provide an attractive framework. They also provide an elegant framework to derive splitting algorithms such as the Alternating Direction Method of Multipliers (ADMM) \cite{lions1979splitting}, or, more recently, primal-dual algorithms \cite{condat2013primal}, and randomized or distributed optimization algorithms \cite{iut-cdc13,wei-ozd-arxiv13,shi-etal-(arxiv)13,iutzeler2013explicit,bianchi2015stochastic}.

In order to accelerate the convergence of fixed point algorithms and in particular optimization methods, there exists a variety of modifications based on the construction of the next iterate by combining the output of the operation with former outputs or iterates. We focus here on the two main modification schemes: \emph{relaxation} and \emph{inertia}.

 {\tiny $\blacksquare$} \emph{Relaxation} combines the output of the operation with the former iterate as
$$ x_{k+1} =  \eta \mathsf{T}(x_k) + (1-\eta) x_k $$
where $\eta$ is some positive parameter. This modification notably appears in Richardson's method for solving linear systems \cite{richardson1911approximate}, and in Krasnoselski\u{\i}--Mann monotone operators convergence theorem. For the gradient algorithm, relaxation amounts to modifying the step-size. For ADMM, the benefits of relaxation are often reduced to the phrase ``experiments [...] suggest that over-relaxation with $\eta \in [1.5,1.8]$ can improve convergence.''  (see \cite{eckstein1992douglas} and \cite[Chap. 3.4.3]{boyd2011distributed}) except in specific cases \cite{ghadimi2013}.

{\tiny $\blacksquare$} \emph{Inertia} on the other side is performed by combining the output of the operation with the former output. An \emph{inertial} iteration for  operator $\mathsf{T}$ writes
\begin{align*}
 \left\{ \begin{array}{l} x_{k+1} =  \mathsf{T}(y_k) \\ y_{k+1} = x_{k} + \gamma(x_{k} - x_{k-1} ) \end{array}  \right.
\Leftrightarrow ~ x_{k+1} =  \mathsf{T}\left( x_{k} + \gamma(x_{k} - x_{k-1} )\right) 
\end{align*}
where $\gamma$ is some positive parameter. This modification was made immensely popular by Nesterov's accelerated gradient algorithm \cite{nesterov1983method}. More recently extensions of this method to proximal gradient (FISTA \cite{beck2009fast}) and ADMM (Fast ADMM \cite{goldstein2014fast}) were proposed and quite popular themselves.

However, despite the popularity of these methods, proving the convergence of the iterates sequence $(x_k)$ is still an issue in many situations (see \emph{e.g.} \cite{chambolle2014convergence,attouch2015rate} for the case of FISTA) and additional restart mechanisms may have to be implemented to improve the convergence properties \cite{o2013adaptive}. Finally, the key problem when using these methods is \emph{tuning efficiently} their parameters. Indeed, ``good'', if not optimal, parameters depend on a variety of elements including the algorithm itself (an optimal parameter for the gradient may make the ADMM divergent for instance) or function parameters in the case of optimization (often through the strong convexity constant which may be hard to estimate \cite{lin2014adaptive} or maladjusted to local analysis \cite{tao2015local}).

\noindent\textbf{Contributions.} In this paper, our aim is to  propose \emph{online acceleration methods} for a general class of fixed point algorithms that encompasses the aforementioned optimization methods. The idea of generic acceleration using inertia was investigated in the sub-linear case in \cite{2015arXiv150602186L} by sequentially solving well-chosen strongly convex approximations of the original problem or in \cite{lin2014adaptive,giselsson2016line} which are based on line-search.

Our approach is to based on the monotone operators framework and more precisely on the \emph{averaging contraction property}, verified by a large class of algorithms such as (proximal) gradient algorithms and, very interestingly, ADMM and recent primal-dual algorithms for which only seldom results exists concerning the choice of relaxation or inertial parameters. More precisely, we begin by considering the particular case of affine operators ($\op{T}(x) = Rx+d$ where $R$ is a matrix and $d$ a vector) and study how theses modifications translate for the spectrum of the linear part and thus for the convergence rate. This spectral characterization makes possible the derivation of optimal parameters (in the linear case) and gives us useful guidelines for the general case. Our online acceleration methods are based on approximating the base algorithm by an affine operator at each iteration and choosing the acceleration parameter as the optimal one for the linear approximation. Finally, we illustrate the performance of our online acceleration methods for the proximal gradient algorithm, ADMM\footnote{For ADMM, this led us to develop a new \emph{Inertial ADMM}, different from Fast ADMM \cite{goldstein2014fast}, build on the monotone operator formulation (see \cite{eckstein1992douglas} and references therein). }, and a primal-dual algorithm on popular lasso and logistic regression problems.

The paper is organized as follows. In Section~\ref{sec:prob}, we introduce the averaged operators framework and related useful lemmas. In Section~\ref{sec:cvreliner}, we formulate {Relaxation}, {Inertia}, and {Alternated Inertia} as modifications on the fixed-point iterations on averaged operators; we provide a coherent set of results concerning convergence (in the general case) and linear rate in the case of affine operators. In Section~\ref{sec:acc}, based on the previous analysis, we derive and prove the convergence of our online acceleration methods. These algorithms are based on the general operator framework and thus fit a large variety of optimization algorithms. Finally, Section~\ref{sec:app} is devoted to numerical illustrations.


\section{Fixed-point Algorithms}
\label{sec:prob}

\subsection{Averaged Operators}

Let $\mathsf{T}$ be a mapping\footnote{For the sake of clarity, we only discuss single-valued mappings in finite dimensional spaces; further results on monotone operators theory can be found in \cite{livre-combettes}. } on $\mathbb{R}^N$.  $\mathsf{T}$ is said \emph{monotone} if   $ \forall x,y\in\mathbb{R}^N$, $
\langle x - y  ; \op{T}(x) - \op{T}(y) \rangle \geq 0$. For $\alpha\in]0,1[$, $\mathsf{T}$ is said \emph{$\alpha$-averaged} iff
\begin{equation*}
\forall x,y\in\mathbb{R}^N, ~~ \|\mathsf{T}(x) - \mathsf{T}(y)\|^2 +\frac{1-\alpha}{\alpha} \| ( \mathsf{I} - \mathsf{T} ) (x) - ( \mathsf{I} - \mathsf{T} ) (y) \|^2 \leq \|x-y\|^2 
\end{equation*}
and $\mathsf{T}$ is said to be \emph{Firmly Non-Expansive (FNE)} if it is $1/2$-averaged. The set of the fixed points of $\op{T}$ will be denoted by $\fix \op{T} = \left\{\bar{x} : \bar{x} = \mathsf{T}\bar{x} \right\}$. 

For instance, if $f$ is a convex function, then its subgradient $\partial f$ is monotonous. $\mathsf{J} = ( I + \partial f )^{-1}$ is FNE.
Furthermore, if its gradient $\nabla f$ is $L$-Lipschitz continuous, $\op{G} = I - (1/L) \nabla f$ is also FNE. Both the fixed points of $\mathsf{J}$ and $\op{G}$ coincide with the points where $0$ belong to $\partial f$ \emph{i.e.} the minimizers of $f$. Similar derivations can be performed for a large class of algorithms such as the proximal gradient, the ADMM, etc.

\begin{lemma}[{Krasnoselski\u{\i}--Mann algorithm}] \emph{\cite[Prop.~5.15]{livre-combettes}}
\label{lem:cv}
Let $\alpha\in]0,1[$. Let $\mathsf{T}$ be an $\alpha$-averaged operator such that $\fix \mathsf{T} \neq \emptyset$. Then, the sequence $(x_k)_{k>0}$ generated by $x^0 \in\mathbb{R}^N$ and the iterations
$$ x_{k+1} = \mathsf{T} (x_k) $$
converges to a point in $\fix \mathsf{T}$.
\end{lemma}

\begin{remark}
\label{rem:fejer}
The iterations produced by averaged operators give \emph{Fej\'er monotonous} iterates sequences $(x_k)_{k>0}$: for any fixed point $\bar{x}$ and iteration $k$, 
$$  \left\| x_{k+1} - \bar{x} \right\| \leq \left\| x_{k} - \bar{x} \right\|; $$
we will investigate this attractive property for the modifications considered.
\end{remark}

\subsection{Linear Convergence of Affine Operators} 
\label{sec:linear} 

We now give a precise characterization of the spectral signification of the averaging property for an affine operator. This will be useful to investigate the relaxation and inertia and will lead to our online algorithms for the general class of averaged operators. 

Results of the literature include analyses of matrices with subdominant eigenvalues and applications to alternating projections and Douglas-Rachford splitting \cite{bauschke2014optimal} or spectral analysis in the case of the FISTA algorithm \cite{tao2015local}. The novelties in our characterization include i) a proof that the algebraic and geometric multiplicities of eigenvalue $1$ coincide, which allows the definition of a proper projection onto the fixed points space (see Apx.~\ref{apx:linear}); ii) the characterization of the practical linear convergence rate based on the greatest eigenvalue in magnitude, $1$ excluded (Theo.~\ref{th:speed}); and iii) the derivation of the position of the eigenvalues under the averaging property (Lemma~\ref{lem:alpha}). For the sake of clarity, all the proof details are reported in Apx.~\ref{apx:linear}.

$\mathsf{T}$ is an \emph{affine operator} denoted by  $\mathsf{T} = R\cdot + d$ if it can be written
$$ \op{T}(x) = Rx + d $$
where $R$ is an $N\times N$ real matrix and $d$ is a size-$N$ real vector.

Let us define the eigenspace of $R$ linked to eigenvalue $1$: $ \mathcal{N} \triangleq \left\{ x \in \mathbb{R}^N : Rx = x \right\} .$
Importantly, as shown in Apx.~\ref{apx:linear}, the averaging property implies that one can define a projection $\Pi_{\mathcal N}$ onto ${\mathcal N}$; and thus the complementary projection $\overline{\Pi_{\mathcal N}}$.

\begin{theorem}
\label{th:speed}
Let $\alpha\in]0,1[$. Let $\mathsf{T} = R \cdot + d$ be an $\alpha$-averaged operator and suppose that $\fix\op{T}\neq\emptyset$. Then, the sequence $(x_k)_{k>0}$ generated by $x^0\in\mathbb{R}^N$ and 
$$ x_{k+1} = \mathsf{T} (x_k) $$
converges linearly to a point in $\fix \mathsf{T}$ at a rate
$$ \nu \triangleq \max \{ |\lambda_i| : \lambda_i\neq 1 \textrm{ is an eigenvalue of } R \}  < 1$$
in the sense that $\exists \bar{x}\in\fix\mathsf{T},   \lim\sup_k \frac{\log \| \overline{\Pi_{\mathcal N}}(x_{k} - \bar{x}) \| }{k}  \leq \log \nu .$
\end{theorem}

The proof is reported in  Apx.~\ref{apx:linear}. In the sequel, we will call any eigenvalue $\lambda$ such that $|\lambda|=\nu$ a \emph{dominant} eigenvalue and we will approximate it online as $v_{k} = \|x_{k+1}-x_k\|/\|x_{k} - x_{k-1}\|$.

\begin{remark}
This definition of the convergence rate differs from \cite{bauschke2014optimal} (notably Example 2.11) as taking the $\log$ enables to retrieve directly the principal eigenvalue and not some $\nu + \varepsilon$ or $k^n \nu^k$; this choice was made in order to match practical rates and justifies our next analysis.
\end{remark}

\begin{lemma}
\label{lem:alpha}
Let $\alpha\in]0,1[$ and $\mathsf{T} \triangleq R\cdot + d$ be an {$\alpha$-averaged} affine operator. Then, every eigenvalue $\lambda_i$ of $R$ satisfies
$ |\lambda_i - (1-\alpha) | \leq \alpha .$ Furthermore, $|\lambda_i| \leq 1 $ with equality iff $\lambda_i = 1$, so $\nu < 1$.
\end{lemma}

This lemma shows, if $\mathsf{T}$ is $\alpha$-averaged, the eigenvalues of $R$ are contained in a disk of center $1-\alpha$ and radius $\alpha$ as illustrated by Fig.~\ref{fig:alpha}-a,b.

\input{Figures/alpha_disk.tex}

\begin{example}[Gradient algorithm on a Quadratic Function]
\label{ex:1}
For a differentiable convex function $f$ with an $L$-Lipschitz gradient $\nabla f$, the standard gradient algorithm writes
\begin{align}
\label{eq:grad}
x_{k+1} = x_k - \frac{1}{L} \nabla f(x_k) 
\end{align}
and the related operator is $\op{T} = \op{I} - \frac{1}{L}\nabla f$.

For this illustration, we take quadratic $f(x) = \frac{1}{2} \| Ax-b \|^2$; thus, $f$ is $L$-smooth with $L = \lambda_{\max}(A^\tran A)$ \emph{and} $\mu$-strongly convex with $\mu = \lambda_{\min}(A^\tran A)$. The iterations are affine and the spectrum of the linear part of $\op{T}$ is comprised in the interval $[0 , 1-\mu/L ]$ so we obtain the well-known rate $\nu = 1-\mu/L$.
\end{example}


\section{Relaxation and Inertia} 
\label{sec:cvreliner} 

In this section, we describe \emph{Relaxation}, \emph{Inertia}, and \emph{Alternated Inertia} as modifications on the classical fixed-point iterations presented above. Notably, we give convergence results for the iterates, and exhibit the differences in monotonicity between inertia and relaxation. In addition, we derive optimal parameters and rates in the case of affine operators with real eigenvalues.

\subsection{Relaxation}
\label{sec:cvrel} 

\subsubsection{Convergence}

For a positive sequence $(\eta_k)$, the \emph{relaxed} iterations follow
\begin{align}
\label{eq:rela1}
x_{k+1} = \eta_k \op{T}(x_k) + (1-\eta_k) x_k  =  \op{T}(x_k) + (\eta_k-1) (\op{T}(x_k) -   x_k) .
\end{align}

As mentioned in the introduction, this modification is present since Richardson's iterations and Krasnosel'ski\u{\i}--Mann algorithm, and over-relaxation ($\eta>1$) is still investigated to improve convergence speed. The following result directly comes from  Krasnoselski\u{\i}--Mann theorem. 

\begin{lemma}
\label{lem:cva}
Let $\alpha\in]0,1[$ and let the sequence $(\eta_k)$ verify $0< \underline{\eta} \leq \eta_k \leq \overline{\eta} < 1/\alpha$ for all $k>0$. Let $\mathsf{T}$ be an $\alpha$-averaged operator such that $\fix \mathsf{T} \neq \emptyset$. Then, the sequence $(x_k)_{k>0}$ generated by $x^0 \in\mathbb{R}^N$ and the iterations
$$ x_{k+1} = \eta_k \op{T}(x_k) + (1-\eta_k) x_k $$
converges to a point in $\fix \mathsf{T}$.
\end{lemma}

The proof is based on the fact that if $\alpha\in]0,1[$ and $\eta\in]0,1/\alpha[$, then $\op{T}_\eta$ is $\eta\alpha$-averaged \cite[Prop.~4.28]{livre-combettes}. The convergence thus directly follows from Lemma~\ref{lem:cv} and the produced iterates are \emph{monotonous} in the light of Remark~\ref{rem:fejer}.

\subsubsection{Optimal parameters for real eigenvalues}
\label{sec:affrel} 

Let $\op{T} = R\cdot + d$ be an $\alpha$-averaged linear operator. Suppose that $R$ has real eigenvalues $\lambda_i\in[1-2\alpha,\lambda] \cup \{1\}$. The eigenvalues of $R_\eta = \eta R + (1-\eta) I$ have the form $\mu_i = \eta \lambda_i + (1-\eta)$ . The effect of over-relaxation (for $\eta>1$) is thus the combination of an inflation and a translation as seen in Figure~\ref{fig:alpha}-c.\\
i) When $\eta>0$ is small enough, the dominant eigenvalue of $R_\eta$ is $ \eta \lambda + (1-\eta) >0$; so that the convergence rate $\nu$ will decrease when $\eta$ increases.\\
ii) When $\eta<1/\alpha$ is big enough, the dominant eigenvalue of $R_\eta$ will be $  \eta(1-2\alpha) + (1-\eta) = 1 - 2\alpha\eta < 0$; so that the convergence rate $\nu$ will increase when $\eta$ increases.\\
Finally, The optimal parameter $\eta^\star$, which minimizes the rate, corresponds to the case where the dominant eigenvalues in the two cases are the opposite one of each other:
\begin{align*}
 \eta^\star = \frac{2}{2\alpha + 1 -\lambda} \textrm{and optimal rate } \nu^\star = \frac{2\alpha -1  + \lambda}{2\alpha + 1 - \lambda} .
\end{align*}

In the field of iterative methods for solving linear systems, relaxation has received a lot of attention and the optimal relaxation boils down to Richardson/Chebyshev iterations (see \cite[Example~4.1]{saad2003iterative} and Fig.~\ref{fig:rellin} for an illustration).

\noindent\emph{{\scshape Application in the setup of Ex.~\ref{ex:1}:}
The relaxed iterations write 
$$ x_{k+1} = x_k - \frac{\eta_{k+1}}{L} \nabla f(x_k) $$
and thus relaxation simply consists in adjusting the step size for the gradient algorithm and we have the following optimal relaxation parameter 
$$ \eta^\star = \frac{2}{1+ \mu/L}  \textrm{ and rate }  \nu^\star =  \frac{ 1 - \mu/L}{1+ \mu/L}$$
leading to an optimal stepsize of $2/(\mu + L)$ which matches the asymptotic optimal stepsize (see e.g. \cite[Sec.~4.1.2]{taylor2015smooth}). }

\subsection{Inertia}
\label{sec:cvine} 

\subsubsection{Convergence}

Stemming from popular inertial methods \cite{polyak1964some,nesterov1983method,nesterov2005smooth}, acceleration techniques based on the use of the memory of the previous outputs are very popular both from a theoretical and a practical point of view (see \cite{tseng2008accelerated} and references therein for an overview of these methods). Formally, with $\op{T}$ an operator, the core of these methods consist in performing the following iterations.
\begin{align}
\label{eq:acc} \left\{ \begin{array}{l} x_{k+1} =  \mathsf{T}(y_k) \\ y_{k+1} = x_{k+1} + \gamma_k (x_{k+1} - x_k ) \end{array}  \right. \Leftrightarrow ~~ x_{k+1} =  \mathsf{T}\left( x_{k} + \gamma_k (x_{k} - x_{k-1} )\right) 
\end{align}

A careful choice of the sequence $(\gamma_k)_{k>0}$  is known to accelerate the theoretical functional convergence rate from $\mathcal{O}(1/k)$ to  $\mathcal{O}(1/k^2)$ for a large class of algorithms (see \cite{tseng2008accelerated,chambolle2014convergence,johnstone2015lyapunov} for details) and is very popular in practice.

However, contrary to the relaxation, this modification of the algorithm deeply changes the algorithm behavior as the error between the iterates and some fixed point is \emph{not monotonously decreasing} anymore which can cause stability or domain problems for the iterates. The next lemma provides a general set of conditions for iterates convergence encompassing several results of the literature  \cite{alvarez2001inertial,alvarez2004weak,mainge2008convergence,lorenz2014inertial,2015arXiv150602186L}.

\begin{lemma}
\label{lem:cvacc}
Let $\alpha\in]0,1[$. Let $\mathsf{T}$ be an $\alpha$-averaged operator such that $\fix \mathsf{T} \neq \emptyset$. Assume one the following:
\begin{itemize}
\item[i)] $\exists \gamma, 0\leq \gamma_k \leq \bar{\gamma} < 1$ and $\sum_{k=1}^\infty \gamma_k \| x_k - x_{k-1} \|^2 < \infty $.
\item[ii)] $\exists \bar{\gamma} < 1$, $(\gamma_k)_{k>0}$ is non-decreasing sequence in $[0, \bar{\gamma})$  such that $\forall k>1$
$$  1-\gamma_{k-1} - (1-\gamma_k) \gamma_k -   \frac{\alpha}{1-\alpha} \gamma_k(1+\gamma_k) \geq \underline{m} > 0.$$
\item[iii)] as a particular case of ii), when $\gamma_k = \gamma$ for all $\forall k>1$,
$ (1-\gamma)^2 > \frac{\alpha}{1-\alpha} \gamma (1+\gamma).$
\end{itemize}
 Then, the sequence $(x_k)_{k>0}$ generated by $x_{-1} = x_0 \in\mathbb{R}^N$ and the iterations
$$ x_{k+1} = \op{T}(x_k + \gamma_k(x_k - x_{k-1}))  $$
converges to a point in $\fix \mathsf{T}$.
\end{lemma}

\subsubsection{Optimal parameters for real eigenvalues}
\label{sec:affiner}

Let us define $\op{T}^\gamma$, the operator generating $(x_{k+1},x_k)$ from $(x_{k},x_{k-1})$ where $x_{k+1} = \op{T}(x_k + \gamma(x_k - x_{k-1}))$. When $\op{T} = {R} \cdot + d$, we have
\begin{align*}
\op{T}^\gamma \left( \left[\begin{array}{c} z_1 \\ z_2 \end{array}\right] \right) = \left[\begin{array}{c} (1+\gamma){R}z_1 - \gamma{R}z_2 + d \\ z_1 \end{array}\right] = \underbrace{ \left[\begin{array}{cc} (1+\gamma){R} & - \gamma{R} \\ {I} & {0}  \end{array}\right]}_{{R}^\gamma} \left[\begin{array}{c} z_1 \\ z_2 \end{array}\right] + \underbrace{\left[\begin{array}{c} d \\ 0 \end{array}\right]}_{\tilde{d}}.
\end{align*}
As for relaxation, the eigenvalues of  $R^\gamma$ can be derived from those of from $R$. However, one eigenvalue $\lambda_i$ of $R$ leads to two eigenvalues for $R^\gamma$; they are the roots of 
$$ p_i(\mu) =  \mu^2 - (1+\gamma)\lambda_i \mu + \gamma\lambda_i. $$ 

The main results are:\\
i) for \emph{negatives eigenvalues} $\lambda_i<0$, the magnitude of $\mu_i$ is
$$ \frac{(1+\gamma) |\lambda_i| + \sqrt{(1+\gamma)^2\lambda_i^2 + 4\gamma|\lambda_i|}}{2} \geq (1+\gamma)|\lambda_i| $$
thus inertia has a \emph{negative effect} on the negative side of the spectrum. For the sake of clarity, we will focus on the non-negative eigenvalue case in the following, corresponding to $\alpha\in(0,1/2]$ for the averaging property.\\
ii) for \emph{non-negative eigenvalues} $\lambda_i\in[0,\lambda]\cup\{1\}$, optimal parameter and rate are
\begin{align*}
 \gamma^\star = \frac{ (1 - \sqrt{1-\lambda} )^2}{\lambda} \text{ and } \nu^{\star} = 1 -\sqrt{1-\lambda}.
\end{align*}
Notably, we have $\nu^{\star} \geq \lambda/2$ which means the rate with inertia cannot be better than half the original rate.

\noindent\emph{{\scshape Application in the setup of Ex.~\ref{ex:1}:}
The inertial iteration of $\op{T}$ writes 
$$ \left\{ 
\begin{array}{rl} 
y_k &= x_k + \gamma_k(x_k - x_{k-1}) \\
x_{k+1} &= y_k - \frac{1}{L} \nabla f(y_k) 
\end{array}
\right.  $$
we have the following optimal inertia parameter 
\begin{align*}
\gamma^\star  = \frac{ 1 - \sqrt{\mu/L} }{1+\sqrt{\mu/L}} \text{ and rate } \nu^\star = 1 -\sqrt{\mu/L}.
\end{align*}
Once again, the obtained parameter and rate matches practical and theoretical optimal situations as summarized in \cite{o2013adaptive}. }

\subsection{Alternated Inertia}

\subsubsection{Convergence}

In order to improve the convergence properties of the iterates of Eq.~\eqref{eq:acc}, it was suggested in \cite{mu2015note} to apply inertia every other iteration. This variant is a lot less popular than vanilla inertia. However, its rather good convergence properties and performances, along with its remarkable closeness with relaxation make it worthy of careful attention. The iterations of alternated inertia are:
\begin{equation}
\label{eq:alacc}
\left\{
\begin{array}{ll}
x_{k+1} = \op{T}(x_k)  & \text{ if } k \text{ is even}  \\
x_{k+1} = \op{T}(x_k + \gamma_k(x_k - x_{k-1}))  & \text{ if } k \text{ is odd}
\end{array}
\right.
\end{equation}

Interestingly, using inertia every other iteration can make the error \emph{monotonously} decreasing again which will reveal to be interesting numerically.

\begin{lemma}
\label{lem:cvalt}
Let $\alpha\in]0,1[$. Let $\mathsf{T}$ be an $\alpha$-averaged operator such that $\fix \mathsf{T} \neq \emptyset$. Assume that the sequence $(\gamma_k)$ verifies $ 0 \leq \gamma_k \leq \frac{1-\alpha}{\alpha}$ for all $k>0$. Then, the sequence $(x_k)_{k>0}$ generated by $ x_0 \in\mathbb{R}^N$ and the iterations
$$ \left\{
\begin{array}{ll}
x_{k+1} = \op{T}(x_k)  & \text{ if } k \text{ is even}  \\
x_{k+1} = \op{T}(x_k + \gamma_k(x_k - x_{k-1}))  & \text{ if } k \text{ is odd}
\end{array}
\right. $$
converges to a point in $ \fix \mathsf{T}$.
\end{lemma}

The proof, which generalizes \cite{mu2015note} to $\alpha$-averaged operators, can be found in Apx.~\ref{apx:cvalt}. 

\subsubsection{Optimal parameters for real eigenvalues}

Let us define $\op{T}^{.,\gamma}$ the operator generating $x_{k+2}$ from $x_k$ for $k$ even.  When $\op{T} = {R} \cdot + d$, one has 
\begin{align*}
x_{k+2} = \op{T} \left( \op{T}(x_k) + \gamma ( \op{T}(x_k) - x_k) \right) = \underbrace{\left[(1+\gamma)R^2 - \gamma R \right]}_{R^{.,\gamma}} x_k + (1+\gamma)Rd + d 
\end{align*}
so that the eigenvalues of $ R^{.,\gamma} $ are $\mu_i = (1+\gamma)\lambda_i^2 - \gamma\lambda_i $ with $\lambda_i$ an eigenvalue of $R$.

The main results are:\\
i) for \emph{negatives eigenvalues} $\lambda_i<0$, the magnitude of $\mu_i$ is
$$ (1+\gamma)|\lambda_i|^2 + \gamma|\lambda_i| $$
thus alternated inertia has a negative effect on negative eigenvalues.\\
ii) for \emph{non-negative eigenvalues} $\lambda_i\in[0,\lambda]\cup\{1\}$, the largest $\mu_i$ in magnitude is linked to the original dominant eigenvalue $\lambda$ \emph{and} intermediate eigenvalues. Taking the worst case scenario over the unknown (and hard to estimate) intermediate eigenvalues, the optimal parameter and rate are
\begin{align*}
 \gamma^\star = \frac{2 (\lambda)^2 + (\sqrt{2}-1)\lambda}{2\lambda(1-\lambda) + \frac{1}{2}} \text{ and } \nu^{\star} = \frac{(\gamma^\star)^2}{4(1+\gamma^\star)}.
\end{align*}

\noindent\emph{{\scshape Application in the setup of Ex.~\ref{ex:1}:}
The alternated inertial iteration of $\op{T}$ writes
$$ \left\{ 
\begin{array}{rl} 
y_k &= x_k + \gamma_k(x_k - x_{k-1}) \\
x_{k+1} &= y_k - \frac{1}{L} \nabla f(y_k) \\
x_{k+2} &= x_{k+1} - \frac{1}{L} \nabla f(x_{k+1})
\end{array}
\right. \Leftrightarrow \left\{ 
\begin{array}{rl} 
x_{k+1} &= x_k - \frac{1}{L} \nabla f(x_k) \\
x_{k+2} &= x_{k+1} - \frac{1+\gamma_{k+2}}{L} \nabla f(x_{k+1})
\end{array}
\right.  $$
This formulation properly illustrates the equivalence between alternated inertia and alternated relaxation. 
We have the following optimal inertia parameter 
\begin{align*}
\gamma^{\star} = \frac{ 2(\mu/L)^2  - (3 + \sqrt{2} )\mu/L + 1 + \sqrt{2}}{ -2(\mu/L)^2  +2\mu/L  + \frac{1}{2} } \text{ and rate } \nu^\star = \frac{(\gamma^{\star})^2}{4(1+\gamma^{\star})}.
\end{align*} }

\subsubsection{Comparison of the optimal rates}

\begin{figure}
\centering
\begin{subfigure}[b]{0.49\columnwidth}
    \centering
	\setlength\figureheight{0.7\columnwidth} 
	\setlength\figurewidth{0.9\columnwidth}
        \input{Figures/comp_grad.tikz}
     \caption{Convergence rate versus the relaxation parameter along with tradeoff enabling to derive optimal relaxation. (for illustration purposes, we assumed $1\geq\lambda_{\max}\geq\lambda_{\min}\geq0$) }\label{fig:rellin}
     \centering
     \end{subfigure}
\begin{subfigure}[b]{0.49\columnwidth}
  \centering
       
	\begin{tikzpicture}[scale = 1]
\begin{axis}[ 
 width=0.8\columnwidth, 
 height=0.7\columnwidth,  
 xmajorgrids, 
 xlabel={$\lambda_{\min}$},
 ymajorgrids, 
 ylabel={$\lambda_{\max}$},
 zlabel={Optimal rate},
zmin = 0,
zmax = 1,
xmin = 0,
xmax = 1,
ymin = 0,
ymax = 1,
 view={20}{35},
 ]

\addplot3[patch,patch type=triangle,color=black, opacity=0.5,faceted color=black] 
file { Figures/comp_r.dat  };

\addplot3[patch,patch type=triangle,color=red, opacity=0.5,faceted color=red] 
file { Figures/comp_i.dat  };

\addplot3[patch,patch type=triangle,color=green!50!black, opacity=0.3,faceted color=green!50!black] 
file { Figures/comp_a.dat  };

\addplot3[patch,patch type=rectangle,color=white, opacity=1.0,faceted color=white] 
coordinates{ 
(-0.04, 0, 0)
(-0.04, 0, 1)
(-0.04, 1, 1) 
(-0.04, 0.8, 0.55) };

\end{axis}
\end{tikzpicture}
    \caption{Best rate obtained with one of the three optimal modifications when the eigenvalues interval $[\lambda_{\min},\lambda_{\max}]$ varies. If the best rate is attained by relaxation, it is displayed in \textbf{black}; by inertia, in {\color{red} \textbf{red}}; by alt. inertia, in {\color{green!50!black} \textbf{green}}.  }\label{fig:lin}
\end{subfigure}
\caption{Effect of studied modification on linear iterations rate}
 \end{figure}

In general, comparison between relaxation and inertia on linear iterations depends on the interval $[\lambda_{\min}, \lambda_{\max}]$ in which the eigenvalues of the original matrix $R$ live.  Fig.~\ref{fig:rellin} provides a graphical illustration of the effect of relaxation on the linear rate. Fig.~\ref{fig:lin} displays a 3D plot of the optimal rate obtained by numerical simulations (the lower the better) when $\lambda_{\min}$ and $\lambda_{\max}$ vary between $0$ and $1$ along with the modification scheme attaining it. 

One can notice that the optimal speed with inertia (alternated or classical) is always faster than with relaxation when $\lambda_{\min} = 0$. For instance, it is the case for the gradient algorithm setup of Ex.~\ref{ex:1}; the equality case being when $\mu/L = 1$. Between alternated and classical inertia, the alternated version is faster for well enough conditioned problems, more precisely when $1 \geq \mu/L \geq 4/(9+4\sqrt{2}) \approx 0.273$; surprisingly making alternated inertia more performing than both inertia and relaxation for some problems.  Also, the optimal parameter for inertia  $\gamma^\star$ is greater than theoretical limit $1/3$ as soon as $\mu/L\leq 1/4$. Similarly, for alternated inertia, optimal $\gamma^\star$ is greater than $1$ when $ \mu/L\leq (3-\sqrt{2})/4 \approx 0.396 $.

Unfortunately, while $L$ can often be known or upper bounded, $\mu$ is in general unknown so that the optimal parameters cannot be computed hence the need for automatically tuned schemes as developed further in this paper.

\section{Online Acceleration of Linear Rates using Relaxation and Inertia} 
\label{sec:acc}

In this Section, we provide practical acceleration algorithms for fixed point iterations of general averaged operators using relaxation and inertia.

These methods, that automatically tune relaxation/inertia parameters, are based on affine approximation with real eigenvalues, as investigated in the previous section. This may appear limiting at first but i) in practice, linear approximation of averaged operators often have dominant eigenvalues close to the real line (real eigenvalues are linked to the \emph{cyclic monotonicity property} which appears when considering (sub)-gradients, see \cite{shiu1976cyclically} and \cite[Theo.~22.14]{livre-combettes}; ii) similar reasoning have been used in recent proofs of inertial algorithms \cite{flammarion2015averaging};  iii) we prove the iterates convergence in the general averaged operator case (not just affine let alone with real eigenvalues) and iv) our method works very well in practice as demonstrated in Section~\ref{sec:app}.

For all three modifications, we will iterate in the same steps:\\ 
From some acceleration parameter $\delta$, \\
i) Apply the accelerated operator $\mathsf{T}_\delta$ on the current iterates and estimate its current rate by computing $v_{k} = \|\mathsf{T}_\delta(x_{k-1})-\mathsf{T}_\delta(x_{k-2})\|/\|x_{k-1} - x_{k-2}\|$ as in Sec.~\ref{sec:linear};\\
ii) From $\delta$ and $v_k$, construct an approximation of the \emph{virtual} dominant eigenvalue $\lambda$ of original operator $\mathsf{T}$ using the results of the previous section (\emph{virtual} as $\mathsf{T}$ is a general non-linear averaged operator, $\lambda$ is thus linked to an affine approximation of $\mathsf{T}$);\\
iii) From $\lambda$, update $\delta$ as the optimal acceleration parameters previously derived.

\subsection{ORM: Online Relaxation Method}
\label{sec:relaxopt}

Building on the derivations of \\Sec.~\ref{sec:affrel}, we wish to estimate $\eta^\star$ without having access to the spectrum of $R$. \\
i) To do so, we estimate the current convergence rate as\footnote{Note the extra factor $\eta_{k-1}/\eta_{k}$ compared to Sec.~\ref{sec:linear}. In the specific case of relaxation, this modified definition enables to estimate the convergence of $\mathsf{T}_{\eta_{k}}$ by applying it only once. Monotone operators theory ensures us that $v_k\in[0,1]$, and enables the convergence proof. }\\  $ v_k = (\eta_{k-1} \|x_{k} - x_{k-1}\| )/(\eta_{k} \|x_{k-1} - x_{k-2} \|) $.\\
ii) Using this $v_k$, the current relaxation $\eta_k$, and the expression for  $\nu^\star$, we can compute an estimate for virtual dominant eigenvalue $\lambda$:   $ \lambda_k = (v_k + \eta_k -1)/\eta_k $.\\
iii) Using $ \lambda_k$ and optimal $\eta^\star$, we take our next relaxation parameter as
$$\eta_{k+1} = \frac{2}{ 2\alpha +1 - \lambda_k } = \frac{2\eta_k}{2\alpha \eta_k + 1 - v_k} .$$

This gives the intuition for our \emph{Online Relaxation Method (ORM)}.

\smallskip
\begin{breakbox}
\noindent  {\bf {Online Relaxation Method} (ORM)  {\small for $\alpha$-averaged operator $\mathsf{T}$} }: \vspace*{-0.3cm} \\
\hrule
\vspace*{0.2cm}
\noindent \underline{Initialization:} $\varepsilon \in ]0,2\min(\alpha;1-\alpha)]$, $x^0$, $x^1 = \mathsf{T} x^0$ ,  $\eta^0=\eta^1=1$. \\
\underline{At each iteration $k\geq1$:}
\begin{align*}
\eta_{k+1} &= \frac{(2 - \varepsilon)\eta_k}{ 2 \alpha \eta_k + 1 - \frac{\eta_{k-1} \|x_{k} - x_{k-1}\|}{ \eta_k \|x_{k-1} - x_{k-2} \|} } + \frac{ \varepsilon}{4\alpha}\\
x_{k+1} &= \eta_{k+1} \mathsf{T} x_{k} + (1-\eta_{k+1})x_k 
\end{align*}   
\end{breakbox}
\smallskip

The following result provides convergence guarantees for this method in the general framework of averaged operators. 

\begin{theorem}
\label{th:orm}
Let $\alpha \in ]0,1[$. Let $\op{T}$ be an $\alpha$-averaged operator such that $\fix \op{T} \neq \emptyset $. Then, the sequence $(x_k)_{k>0}$ generated by the Online Relaxation Method converges to a point in $\fix \op{T}$.
\end{theorem}
\begin{proof}
In order to use Lemma~\ref{lem:cva} to prove the convergence, let us prove by induction that for all $k\geq 1$, $\eta_k \in [ \frac{\varepsilon }{4\alpha} , \frac{1}{\alpha} -  \frac{\varepsilon }{4\alpha} ]$. It is obviously true for $\eta^1 = 1$. Let us assume that $\eta_k \in [ \frac{\varepsilon }{4\alpha} , \frac{1}{\alpha} -  \frac{\varepsilon }{4\alpha} ]$.\\
First, as $\op{T}$ is $\alpha$-averaged, it writes $\op{T} = \alpha \op{R} + (1-\alpha)\op{I}$ with $\op{R}$ a non-expansive operator and $\op{I}$ the identity. $\op{T}_{\eta_k}$ then writes $\op{T}_{\eta_k} = \alpha \eta_k \op{R} + (1-\alpha\eta_k)\op{I}$, thus
\begin{align*}
\|&x_{k} - x_{k-1} \| = \alpha\eta_k \|\op{R} (x_{k-1}) - x_{k-1} \| \\
&=  \alpha\eta_k \|\op{R} (x_{k-1}) - \op{R}(x_{k-2}) + (1-\alpha \eta_{k-1})(\op{R}(x_{k-2}) - x_{k-2}  ) \| \\
&\leq \alpha\eta_k \|x_{k-1} - x_{k-2} \| + \alpha\eta_k (1-\alpha \eta_{k-1}) \|\op{R}(x_{k-2}) - x_{k-2} \| \\
&=  \alpha\eta_k \|x_{k-1} - x_{k-2} \| + \alpha\eta_k \frac{1-\alpha \eta_{k-1}}{\alpha \eta_{k-1}} \|x_{k-1} - x_{k-2} \| =  \frac{\eta_k}{\eta_{k-1}} \|x_{k-1} - x_{k-2} \|
\end{align*}

Thus, we have $ v_k \leq 1$ which makes $\eta_{k+1} \geq   \frac{ \varepsilon}{4\alpha} $. Now,
\begin{align*}
\eta_{k+1} \leq \frac{(2-\varepsilon) \eta_k}{ 2 \alpha \eta_k  } + \frac{ \varepsilon}{4\alpha} = \frac{1}{\alpha}  - \frac{\varepsilon }{2\alpha} + \frac{\varepsilon }{4\alpha} = \frac{1}{\alpha}  - \frac{\varepsilon }{4\alpha}
\end{align*}
thus we have that $\eta_{k+1} \in [ \frac{\varepsilon }{4\alpha} , \frac{1}{\alpha} -  \frac{\varepsilon }{4\alpha} ]$. This means the generated sequence $(\eta_k)_{k>0}$ lies in $[ \frac{\varepsilon }{4\alpha} , \frac{1}{\alpha} -  \frac{\varepsilon }{4\alpha} ]$ and thus verifies the conditions of Lemma~\ref{lem:cva} for convergence.
\end{proof}

Interestingly, one can notice that when the basis algorithm converges sub-linearly, the paramter chosen by ORM becomes close to $2/L$. For the gradient algorithm, this would amount to having a stepsize that becomes close to $2/L$ as the number of iterations grow which is coherent with the optimality results in \cite[Sec.~4.1.1]{taylor2015smooth}\footnote{The optimal stepsize when doing $K$ iterations goes to $2/L$, staying strictly below, when $K\to\infty$.}.

\subsection{OIM: Online Inertia Method}
\label{sec:inertiaopt}

An online inertia method can be proposed based on the same principles as ORM building on Sec.~\ref{sec:affiner}. In the same vein, we approximate the operation $\op{T}^{\gamma_{2k}}$ by an affine operator with non-negative eigenvalues.\\
i) We estimate the current convergence rate related to operator $\op{T}^{\gamma_{2k}}$ (by applying twice the same inertia twice) as\\ $v_{2k} = \sqrt{ \frac{ \|x_{2k+2} - x_{2k+1} \|^2 +  \|x_{2k+1} - x_{2k} \|^2 }{ \|x_{2k+1} - x_{2k} \|^2 +  \|x_{2k} - x_{2k-1} \|^2} } $.\\
ii) Using $v_{2k}$, current inertia $\gamma_{2k}$, we estimate $\lambda$:  $ \lambda_{2k} = ((v_{2k})^2)/(\gamma_{2k} v_{2k} - \gamma_{2k} + v_{2k})  $.\\
iii) Using $ \lambda_{2k}$ and the formula for optimal $\gamma^\star$, we take our next relaxation parameter as
$\gamma_{2k+2} = (1 - \sqrt{1-\lambda_{2k}})^2/\lambda_{2k}$.

These steps are at the core of our \emph{Online Inertia Method (OIM)}. However, to the difference of ORM but similarly to other inertia-based accelerations \cite{goldstein2014fast}, a restart mechanism has to be introduced to make sure the algorithm converges. Indeed, this scheme, which is rather aggressive, often overpasses the theoretical limits of the convergence results. Thus, in order to maintain convergence, the algorithm must either i)  sufficiently decrease the error $\|x_k - y_k\|$; or ii) set inertial parameter $\gamma_k$ to $0$ so that classical convergence results apply.

\smallskip
\begin{breakbox}
\noindent  {\bf Online Inertia Method (OIM)  {\small for $\alpha$-averaged operator $\mathsf{T}$}}:\vspace*{-0.3cm} \\
\hrule
\vspace*{0.2cm}
\noindent \underline{Initialization:} $x_{1}$, $x_2 = \mathsf{T}(x_1)$, $y_2=x_1$ $\gamma_2 = 0$, $\varepsilon>0$. \\
\underline{For each $k\geq 1$:}\\
\hspace*{1.5cm}$\left\{ \begin{array}{rl} y_{2k+1} &= x_{2k} + \gamma_{2k}(x_{2k} - x_{2k-1}) \\
x_{2k+1} &= \op{T}(y_{2k+1}) \\
y_{2k+2} &= x_{2k+1} + \gamma_{2k}(x_{2k+1} - x_{2k}) \\
x_{2k+2} &= \op{T}(y_{2k+2}) \end{array} \right. $\\[0.2cm]
\hspace*{1.5cm} $c_{2k} = \max\left( \frac{\|x_{2k+2}-y_{2k+2}\|}{\|x_{2k+1}-y_{2k+1}\|}  ; \frac{\|x_{2k+1}-y_{2k+1}\|}{\|x_{2k}-y_{2k}\|}  \right)$\\[0.2cm]
\hspace*{1.5cm} \textbf{if}  $ c_{2k} \leq 1 - \varepsilon $   \hspace*{2.5cm} {\small [Acceleration] }\\[0.2cm]
\hspace*{2.5cm} $v_{2k} = \sqrt{ \frac{ \|x_{2k+2} - x_{2k+1} \|^2 +  \|x_{2k+1} - x_{2k} \|^2 }{ \|x_{2k+1} - x_{2k} \|^2 +  \|x_{2k} - x_{2k-1} \|^2} } $\\
\hspace*{2.5cm} $\lambda_{2k} =  \min\left(  \frac{(v_{2k})^2}{\gamma_{2k} v_{2k} - \gamma_{2k} + v_{2k}} ; 1-\varepsilon \right)  $\\
\hspace*{2.5cm} $\gamma_{2k+2} =  \max\left(0 ; \frac{ (1 - \sqrt{1-\lambda_{2k}})^2}{\lambda_{2k}} \right)$\\
\hspace*{1.5cm} \textbf{elseif }  $ \gamma_{2k} > 0$   \hspace*{2.3cm} {\small [Restart] } \\
\hspace*{2.5cm} $\gamma_{2k+2} =  0 $\\
\hspace*{2.3cm} {\small $(x_{2k+1} , x_{2k+2} , y_{2k+2} ) = (x_{2k-1} ,  x_{2k} , y_{2k})$}\\
\hspace*{1.5cm} \textbf{elseif  }  $ \gamma_{2k} = 0$ \hspace*{2.3cm} {\small [No Acceleration] }\\
\hspace*{2.5cm} $\gamma_{2k+2} =  0 $\\
\end{breakbox}
\smallskip

\begin{theorem}
\label{th:oim}
Let $\alpha \in ]0,1[$. Let $\op{T}$ be an $\alpha$-averaged operator such that $\fix \op{T} \neq \emptyset $. Then, the sequence $(y_k)_{k>0}$ generated by the Online Inertia Method converges in the sense that $\|\op{T}(y_k) - y_k \| \to 0$. Furthermore, if $\fix \op{T}$ is reduced to a single point $x^\star$, $x_k \to x^\star$.
\end{theorem}
\begin{proof}
The proof follows the same reasoning as \cite[Theo. 3]{goldstein2014fast}. At each iteration, one of the following situation happens:\\
\emph{i)} the last iteration was beneficial: $c_{k} \leq 1 - \varepsilon$ so that $\|x_{k}-y_{k}\| \leq (1-\varepsilon) \|x_{k-1}-y_{k-1}\|  $ and $\|x_{k-1}-y_{k-1}\| \leq (1-\varepsilon) \|x_{k-2}-y_{k-2}\|  $ ;\\
\emph{ii)} a restart is made so that the iterates $x_{k}$ and $x_{k-1}$ by their previous values $\|x_{k}-y_{k}\| = \|x_{k-2}-y_{k-2}\| $ and $\|x_{k-1}-y_{k-1}\| = \|x_{k-3}-y_{k-3}\| $;\\
\emph{iii)} there is no acceleration and non expansiveness gives $\|x_{k}-y_{k}\|  \leq \|x_{k-1}-y_{k-1}\|  \leq \|x_{k-2}-y_{k-2}\| $.\\
To conclude the proof, one has to notice that for all $k>0$,  $\|x_{k}-y_{k}\|  \leq \|x_{k-1}-y_{k-1}\|$ and $\|x_{k}-y_{k}\|  \leq (1-\varepsilon) \|x_{k-1}-y_{k-1}\|$ if \emph{i)} happens. Now, if there is a finite number of beneficial iterations (when \emph{i)} happens), then after the last one, the algorithm goes back to the unaccelerated iterations and convergence is ensured by Lemma~\ref{lem:cv}. If there is a infinite number of beneficial iterations, introducing variable $\iota_k$ as $\iota_k = 1$ if iteration $k$ is beneficial and $0$ elsewhere; we have
\begin{align*}
\sum_{k=1 }^{\infty}  \iota_k \| \op{T}(y_k) - y_k \| \leq  \| \op{T}(x^0) - x^0  \| \sum_{k=1}^{\infty} \prod_{\ell=1}^k (1-\varepsilon)^{\iota_\ell} \leq  \frac{\| \op{T}(x^0) - x^0  \|}{\varepsilon} < \infty
\end{align*}
and thus $\| \op{T}(y_k) - y_k \| \to 0$. This means that the accumulation points of $(y_k)$ are in $\fix \op{T}$. In addition, if it is reduced to a single point, then $(y_k)$ converges to it and as $\|x_k - y_k \| \to 0$, so does $(x_k)$. 
\end{proof}


When the convergence is sublinear, the restart condition based on a constant $\varepsilon$ may be too harsh. Following the convergence proof, one can easily deduce that $\varepsilon$ can be taken as a sequence $(\varepsilon^\ell)$ provided that $1/\varepsilon^\ell = o(\ell) $ where $\ell$ is the number of accelerations. For instance, a typical setting is to keep track of the number of accelerations $\ell$ and take $\varepsilon^\ell = \varepsilon_0/\sqrt{\ell}$.  Note that in the sublinear case, the OIM makes the acceleration parameter go to $1$ as in Nesterov's optimal method \cite{nesterov1983method}.

\subsection{OAIM: Online Alternated Inertia Method}
\label{sec:ainertiaopt}

Using the same reasoning as for OIM, we are able to obtain a similar algorithm.

\smallskip
\begin{breakbox}
\noindent  {\bf Online Alternated  Inertia Method (OAIM) {\small for an $\alpha$-averaged operator $\mathsf{T}$}}: \vspace*{-0.3cm} \\
\hrule
\vspace*{0.2cm}
\noindent \underline{Initialization:} $x_{3}$, $x_4 = \mathsf{T}(x_3)$, $y_4=x_3$ $\gamma_4 = 0$, $\varepsilon>0$. \\
\underline{For each $k\geq 1$:}\\
$\left\{ \begin{array}{rl} y_{4k+1} &= x_{4k} + \gamma_{4k}(x_{4k} - x_{4k-1}) \\
x_{4k+1} &= \op{T}(y_{4k+1}) \\
y_{4k+2} &= x_{4k+1} \\
x_{4k+2} &= \op{T}(y_{2k+2}) \\  \end{array} \right. ~~~ \left\{ \begin{array}{rl}  y_{4k+3} &= x_{4k+2} + \gamma_{4k}(x_{4k+2} - x_{4k+1}) \\
x_{4k+3} &= \op{T}(y_{4k+3}) \\
y_{4k+4} &= x_{4k+3} \\
x_{4k+4} &= \op{T}(y_{4k+4}) \end{array} \right.$\\[0.2cm]
\hspace*{1.5cm} $c_{4k} = \max\left( \frac{\|x_{4k+4}-x_{4k+3}\| }{ \|x_{4k+2}-x_{4k+1}\|} ; \frac{\|x_{4k+2}-x_{4k+1}\| }{ \|x_{4k}-x_{4k-1}\|} \right) $\\[0.2cm]
\hspace*{1.5cm} \textbf{if}  $ c_{4k} \leq 1 - \varepsilon $   \hspace*{2.3cm} {\small [Acceleration] }\\
\hspace*{2cm} $v_{4k} = \frac{\|x_{4k+4}-x_{4k+2}\| }{ \|x_{4k+2}-x_{4k}\|} $\\
\hspace*{2cm} $\lambda_{4k} = \min\left( \frac{\gamma_{4k} + \sqrt{ (\gamma_{4k})^2 + 4 \gamma_{4k} v_{4k} + 4v_{4k}} }{2(\gamma_{4k} +1)}; 1 - \varepsilon \right) $\\
\hspace*{2cm} $\gamma_{4k+4} =  \frac{ 2(\lambda_{4k})^2 + (\sqrt{2}-1)\lambda_{4k}}{2\lambda_{4k}(1-\lambda_{4k}) + \frac{1}{2}}  $\\
\hspace*{1.5cm} \textbf{elseif }  $ \gamma_{4k} > 0$   \hspace*{2.3cm} {\small [Restart] } \\
\hspace*{2cm} $\gamma_{4k+4} =  0 $\\
\hspace*{2cm} {\small $(x_{4k+3} , x_{4k+4} ) = (x_{4k-1} ,  x_{4k} )$}\\
\hspace*{1.5cm} \textbf{elseif  }  $ \gamma_{4k} = 0$ \hspace*{2.3cm} {\small [No Acceleration] }\\
\hspace*{2cm} $\gamma_{4k+4} =  0 $
\end{breakbox}
\smallskip

\begin{theorem}
Let $\alpha \in ]0,1[$. Let $\op{T}$ be an $\alpha$-averaged operator such that $\fix \op{T} \neq \emptyset $. Then, the sequence $(x_k)_{k>0}$ generated by the Online Alternated Inertia Method converges in the sense that $\|\op{T}(x^{2k}) - x^{2k} \| \to 0$. Furthermore, if $\fix \op{T}$ is reduced to a single point $x^\star$, $x_{k} \to x^\star$.
\end{theorem}
\begin{proof}
The proof follow the same steps as the proof of Theo.~\ref{th:oim}.
\end{proof}


\section{Relaxation and Inertia of Optimization algorithms} 
\label{sec:app} 

We now particularize the operator $\op{T}$ to different values corresponding to popular algorithms of the literature. We illustrate the interest of the modifications studied and, most importantly, we demonstrate the acceleration provided by our online methods over three popular algorithms: the proximal gradient algorithm, the ADMM, and a Primal-Dual algorithm  by Condat \cite{condat2013primal}.

For each of these algorithms, we will proceed in the same fashion:\\
\emph{1)} We discuss how relaxation and inertia translate for these algorithms along with a review on existing accelerated versions;\\
\emph{2)} We provide numerical illustrations over the three following functions chosen for their differences in terms of smoothness and strong convexity:
\begin{itemize}
\item[\emph{a)}] lasso:
\begin{align*}
\min_{x\in\mathbb{R}^n} F_a(x) = \underbrace{\frac{1}{2} \left\| Ax-b \right\|_2^2}_{f_a(x)} +  \underbrace{\lambda \|x\|_1}_{g_a(x)}
\end{align*}
where $A$ has $m = 100$ examples and $n=300$ observations taken from the normal distribution with zero mean and unit variance, the columns of $A$ are then scaled to have unit norm. $b$ is generated by i) drawing a sparse vector $p \in \mathbb{R}^n$ with $90$ non-zeros entries taken from  the normal distribution with zero mean and unit variance; ii) then creating $b$ as  $b  =  A p + e $ where $e$ is a small white noise taken from the normal distribution with zero mean and standard deviation $\sigma = 0.001$. $\lambda$ is chosen so that the optimal solution has sought sparsity. Lipschitz constant of $\nabla f_a$ is taken equal to true $L=\|A^\mathrm{T} A\|_2$.
\item[\emph{b)}] $\ell_1$-regularized logistic regression:
\begin{align*}
\min_{x\in\mathbb{R}^n} F_b(x) = \underbrace{\frac{1}{m} \sum_{i=1}^m \log\left(1+\exp(-y_i \langle a_i,x\rangle)\right) }_{f_b(x)} +  \underbrace{\lambda \|x\|_1}_{g_b(x)}
\end{align*}
where the couples class/feature vector $(y_i,a_i)\in\{-1,1\}\times \mathbb{R}^n$ are taken from the \texttt{ionosphere} binary classification dataset\footnote{\url{https://archive.ics.uci.edu/ml/datasets/Ionosphere}} which has $m=351$ observations and $n=34$ features. Each feature was normalized to have zero mean and unit variance, the resulting size-$n$ observation vectors are denoted by $(a_i)_{i=1,..,m}$ and the observed classes ${-1,+1}$ are denoted by $(y_i)_{i=1,..,m}$. Lipschitz constant of $\nabla f_b$ is upper bounded by $L'  = \max_i \|a_i\|_2^2 $. $\lambda$ was taken equal to $0.1$.
\item[\emph{c)}] $\ell_2$-regularized logistic regression:
\begin{align*}
\min_{x\in\mathbb{R}^n} F_c(x) = \underbrace{\frac{1}{m} \sum_{i=1}^m \log\left(1+\exp(-y_i \langle a_i,x\rangle)\right) }_{f_c(x)} +  \underbrace{\frac{\lambda}{2} \|x\|_2^2}_{g_c(x)}
\end{align*}
where the couples class/feature vector $(y_i,a_i)\in\{-1,1\}\times \mathbb{R}^n$ are taken from the same dataset, and $\lambda$ was taken equal to $0.01$.
\end{itemize}
For these three functions we computed approximated optimal values by external solvers. For the online algorithms, the convergence-ensuring $\varepsilon$ is set to $10^{-4}$

\subsection{Proximal Gradient Algorithm} ~\\

{\small
\noindent\begin{minipage}{0.98\columnwidth}
{\bfseries Proximal Gradient algorithm for $ \min_x f(x) + g(x) $, $f$ $L$-smooth.  }  \hrulefill
\begin{align*}
x_{k+1} &= \argmin_w \left\{ g(w) + \frac{L}{2} \left\| w  - x_k + \frac{1}{L} \nabla f(x_k) \right\|^2 \right\} 
\end{align*}
\hrule
\end{minipage}}
\smallskip

\subsubsection{Accelerations}

It is straightforward to see that an iteration of the algorithm writes as fixed point iteration $x_{k+1} = \mathsf{T}_{pg}(x_k)$ and monotone operator theory tells us that $\mathsf{T}_{pg}$ is $2/3$-averaged \cite[Chap.~27.3]{livre-combettes}. The application of both relaxation and inertia on top of thus algorithm is thus easy. 

In fact, the proximal gradient algorithm possesses a very popular inertial version with the popular FISTA method (proximal gradient + Nesterov's acceleration) \cite{beck2009fast}.

\subsubsection{Numerical Illustrations}

In Fig.~\ref{fig:ista}, we plot the functional error and the parameters for i) classical proximal gradient algorithm; ii) FISTA; iii) our three online methods (we implemented OIM and OAIM as if $\alpha$ was $1/2$, in coherence with the choice for FISTA). We observe that all proposed algorithms show good behaviors, the less favorable case being \emph{b} as neither functions exhibit strong convexity. Inertia-based methods perform very well: OIM outperforms FISTA except in case \emph{b} and OAIM performs quite well. We notice significant behavioral differences between inertial methods (OIM, FISTA) which show bounces in the error descent, contrary to OAIM which is much more stable with almost no use of restart, and ORM which is provably monotonous.

\begin{figure}
    \centering
    \begin{subfigure}{\textwidth}
        \input{Figures/pg_lasso.tikz}
	\input{Figures/pg_lassoP.tikz}
	\caption{lasso  $F_a$}
    \end{subfigure}\\
    \begin{subfigure}{\textwidth}
        \input{Figures/pg_l1.tikz}
	\input{Figures/pg_l1P.tikz}
	\caption{ $\ell_1$-regularized logistic regression $F_b$}
    \end{subfigure}\\
    \begin{subfigure}{\textwidth}
        \input{Figures/pg_l2.tikz}
	\input{Figures/pg_l2P.tikz}
	\caption{$\ell_2$-regularized logistic regression  $F_c$}
    \end{subfigure}
    \caption{Proximal Gradient}\label{fig:ista}
\end{figure}

\subsection{Alternating Direction Method of Multipliers}

Consider the following optimization problem:
\begin{align}
\label{eq:pb-split}
 \min_{x\in\mathbb{R}^N} f(x) + g(Mx) 
\end{align}
with $f,g$ two convex lower semi-continuous functions and $M$ a linear operator. The Alternating Direction Method of Multipliers (ADMM) addresses this problem by performing the following iterations with free parameter $\rho>0$.

{\small
\noindent\begin{minipage}{0.98\columnwidth}
{\bfseries ADMM} \hrulefill
\begin{align*}
u_{k+1} &= \argmin_w \left\{ f(w) + \frac{\rho}{2} \left\| Mw  - v_k + \frac{\lambda_k}{\rho} \right\|^2 \right\} \\
v_{k+1} &= \argmin_w \left\{ g(w) + \frac{\rho}{2} \left\| Mu_{k+1}  - w + \frac{\lambda_k}{\rho} \right\|^2 \right\} \\
\lambda_{k+1} &= \lambda_k  + \rho(Mu_{k+1} - v_{k+1})
\end{align*}
\hrule
\end{minipage}
\smallskip} 


\subsubsection{Accelerations}
\label{ref:admmacc}

From an operator point of view, the iterations of ADMM can be seen as updates on the meta-variable $x_k = \lambda_k + \rho v_k = \lambda_{k-1} + \rho Mu_k $ of an $1/2$-averaged operator $\op{T}_{admm}$ (see \cite{eckstein1992douglas} and references therein for details). This meta-variable is central as it affects the way relaxation and inertia translates for this algorithm.  

\textbf{Relaxation}  While it is fairly evident to see that the relaxed version of the operation writes $x_{k+1} = \eta \op{T}_{admm}(x_k) + (1-\eta)x_k$, it is slightly more complex to derive the effect of relaxation on the algorithm variables ($u_k,v_k,\lambda_k$). Indeed, these variables are computed by a \emph{representation} of the meta-variable that is non-linear. Let us call $\op{J}_v$ the operation giving $v_k$ from $x_k$, then 
{\small
\begin{align}
\label{eq:outer}
v_{k+1} = \op{J}_v (x_{k+1}) = \op{J}_v( \eta \op{T}_{admm}(x_k) + (1-\eta)x_k )  \neq \eta \op{J}_v( \op{T}_{admm}(x_k)) + (1-\eta) \op{J}_v( x_k ).
\end{align}}

This means that, in general\footnote{If either i) $g$ is the indicator function of a linear space, or ii) when $g$ is quadratic; then the representation operation $\op{J}_v$ of Eq.~\eqref{eq:outer} becomes linear and relaxation can be performed as an \emph{outer} modification.} relaxation \emph{cannot} be added directly \emph{on top} of ADMM in the sense performing the standard ADMM update then adding a step of the form $v_{k+1} \leftarrow \eta v_{k+1}  + (1-\eta) v_k$ and $\lambda_{k+1} \leftarrow \eta \lambda_{k+1}  + (1-\eta) \lambda_k$.

Following the operator vision, the canonical relaxation on the ADMM leads to the following iterations (derivations can be found in \cite{eckstein1992douglas}); with $x_k$ being the meta-variable that is F\'ejer monotonous, and used in ORM for instance.

{\small
\noindent\begin{minipage}{0.98\columnwidth}
{\bfseries Relaxed ADMM} \hrulefill
\begin{align*}
u_{k+1} &= \argmin_w \left\{ f(w) + \frac{\rho}{2} \left\| Mw  - z_k + \frac{\lambda_k}{\rho} \right\|^2 \right\} \\
v_{k+1} &= \argmin_w \left\{ g(w) + \frac{\rho}{2} \left\| \eta Mu_{k+1} + (1-\eta)v_k  - w + \frac{\lambda_k}{\rho} \right\|^2 \right\} \\
\lambda_{k+1} &= \lambda_k  + \rho( \eta Mu_{k+1} + (1-\eta)v_k  - v_{k+1}) \\
x_{k+1} =  \lambda_{k+1} + \rho v_{k+1} &
\end{align*}
\hrule
\end{minipage}
\smallskip}

It was noted in  \cite{eckstein1992douglas} that ``experiments [...] suggest that over-relaxation with $\eta \in [1.5,1.8]$ can improve convergence'' without further details. One can also mention \cite{giselsson2016line} based on relaxation tuning by line search. Our ORM, with its particularly stable behavior bridges nicely the literature in this respect.

\textbf{Inertia} As previously, \emph{inertial ADMM} cannot be derived simply by adding inertia \emph{on top} of the above iterations. Following the operator vision, the canonical inertial version of the ADMM leads to the following iterations; with $x_k$ and $y_k$ being the meta-variables as in Eq.~\eqref{eq:acc}. See Apx.~\ref{apx:inadmm} for the derivation. To the best of our knowledge, this is an original algorithm.

{\small
\noindent\begin{minipage}{0.98\columnwidth}
{\bfseries Inertial ADMM} \hrulefill
\begin{align*}
u_{k+1} &= \argmin_w \left\{ f(w) + \frac{\rho}{2} \left\| Mw  - v_k + \frac{\lambda_k}{\rho} \right\|^2 \right\} \\
x_{k+1} = \lambda_k &+ \rho Mu_{k+1}  \\
v_{k+1} &= \argmin_w \left\{ g(w) + \frac{\rho}{2} \left\|  Mu_{k+1} - w + \frac{\lambda_k}{\rho}  + \gamma \left( M (u_{k+1} - u_k) + \frac{\lambda_{k} - \lambda_{k-1}}{\rho} \right) \right\|^2 \right\} \\
\lambda_{k+1} &= \lambda_k  + \rho(  Mu_{k+1}  - v_{k+1})  +  \gamma \rho \left(  M (u_{k+1} - u_k)+ \frac{\lambda_{k} - \lambda_{k-1}}{\rho} \right) \\
y_{k+1} =  \lambda_{k+1} &+ \rho v_{k+1} 
\end{align*}
\hrule
\end{minipage}}

As for relaxation, if  $g$ is either i) the indicator function of a linear space, or ii) quadratic; inertia can be performed as an \emph{outer} modification. Note that ADMM + outer inertia with Nesterov-like parameter sequence corresponds to the algorithm named \emph{Fast ADMM} studied in \cite{goldstein2014fast}. However, this algorithm is not convergent in the general case, unless a restart scheme is added. Interestingly, for the convergence proof of Fast ADMM in the strongly convex case, $g$ is assumed quadratic.

\textbf{Alternated Inertia} It simply consists in alternating an iteration of ADMM with an iteration of Inertia ADMM. One can remark, that with this proper formulation of relaxed and inertial ADMM, applying inertia or relaxation every other iteration provably gives the same algorithm (use the fact that $\lambda_{k} - \lambda_{k-1} = (Mu_k-v_k)/\rho$ for a standard ADMM iteration in the Inertial ADMM iteration). To the best of our knowledge, this kind of algorithm has never been considered before.

\subsubsection{Numerical illustrations}

In Fig.~\ref{fig:admm}, we compare i) the standard ADMM; ii) our three proposed online methods; and iii) Fast ADMM with restart \cite{goldstein2014fast}.  In all cases, the ADMM parameter $\rho$ was set to $1$. For logistic regression functions $f_b$ and $f_c$, no explicit formulation of the update of the first variable is available so their have to be computed by an external solver (\texttt{SciPy}'s general-purpose \texttt{minimize} function in our case). 
We observe that, once again, the proposed online methods show remarkable performance for their computational cost. OIM performs best; however, ORM and OAIM, contrary to OIM and Fast ADMM show \emph{steady} parameter sequences, this can be seen as \emph{more monotonous behaviors}. Finally, ORM offers a better alternative to arbitrarily fixed relaxation.

\begin{figure}
    \centering
    \begin{subfigure}{\textwidth}
        \input{Figures/admm_lasso.tikz}
	\input{Figures/admm_lassoP.tikz}
	\caption{lasso  $F_a$}
    \end{subfigure}\\
    \begin{subfigure}{\textwidth}
        \input{Figures/admm_l1.tikz}
	\input{Figures/admm_l1P.tikz}
	\caption{ $\ell_1$-regularized logistic regression $F_b$}
    \end{subfigure}\\
    \begin{subfigure}{\textwidth}
        \input{Figures/admm_l2.tikz}
	\input{Figures/admm_l2P.tikz}
	\caption{$\ell_2$-regularized logistic regression  $F_c$}
    \end{subfigure}
    \caption{Alternating Direction Method of Multipliers (ADMM)}\label{fig:admm}
\end{figure}



\subsection{a Primal Dual Algorithm}
We investigate the primal-dual algorithm 3.1 from \cite{condat2013primal} with $F=0$. For this algorithm, we will consider only\footnote{for the other two problems, the algorithm boils down to previously investigated ADMM.} the lasso problem $F_a$ as it can be implemented so that, contrary to the ADMM, no matrix inversion is performed, with $M=A$, $g(\cdot) = 1/2 \| \cdot - b\|^2 $ and $f(\cdot) = \lambda \|\cdot\|_1 = g_a(\cdot)$. We chose $\tau=0.5$ and $\sigma= 1/(\tau \|A\|^2)$ as prescribed. 

{\small
\noindent\begin{minipage}{0.98\columnwidth}
{\bfseries a Primal-Dual algorithm \cite[Alg. 3.1]{condat2013primal} for $ \min_x f(x) + g(Mx) $ }  \hrulefill
\begin{align*}
u_{k+1} &= \argmin_w \left\{ f(w) + \frac{1}{2\tau} \left\| w  - u_k + \tau M^{\textrm{T}} \lambda_k \right\|^2 \right\} \\
\lambda_{k+1} &= \lambda_k + \sigma M (2u_{k+1} - u_k) - \sigma \argmin_w \left\{ h(w) + \frac{\sigma}{2} \left\| w - \frac{\lambda_k}{\sigma}  -M (2u_{k+1} - u_k)\right\|^2 \right\} 
\end{align*}
\hrule
\end{minipage}}
\smallskip

With the prescribed choice of parameters, defining $x_{k} = [u_k;\lambda_k]$ as the stacked vector of the variables, the algorithm is a fixed point algorithm on $x_k$ with an $1/2$-averaged operator. Relaxation and Inertia can be simply performed as outer-modifications of the algorithm.

In Fig.~\ref{fig:pd}, we plot the functional error and the parameters for the original algorithm and our three online methods. The formulation of all algorithms are again quite simple and we obtain significant speed improvements.

\begin{figure}
    \centering
        \input{Figures/pd.tikz}
	\input{Figures/pd_step.tikz}
    \caption{a Primal-Dual algorithm on a lasso problem}\label{fig:pd}
\end{figure}


\section{Conclusion} 
\label{sec:cl} 

In this paper, we investigated the theoretical and practical interests of relaxation and inertia on averaged operators. Notably, we established the expression for optimal parameters and rate when possible and built upon it to propose novel online methods. Numerical illustrations have demonstrated the behavioral differences between relaxation and inertia and showed the remarkable performance of the proposed online methods.

\bibliographystyle{siamplain}
\bibliography{math}

\appendix

\newpage

\section{Proof of the linear behavior of affine averaged operators (Sec.~\ref{sec:linear})}
\label{apx:linear}

We consider the fixed point iterations $x_{k+1} = \mathsf{T}( x_k) = Rx_k +d$ with $\mathsf{T} = R\cdot + d $ an affine $\alpha$-averaged operator. We assume that  $\fix\mathsf{T} \neq \emptyset$ that is, $d$ lives in the column space of $I-R$. 

Let us denote by $\mathcal{N}$ the nullspace of $I-R$: $\mathcal{N} \triangleq \left\{ x \in \mathbb{R}^N : Rx = x \right\} $. Any fixed point of $\mathsf{T}$ can be expressed as one particular fixed point plus a vector in $\mathcal{N}$. 

Consider the Jordan decomposition of matrix $R$: $R = W \Lambda W^{-1}$ with $W$ a non-singular matrix and $\Lambda$ the Jordan block-diagonal for $R$ (see \cite[Chap.~3]{HorJoh94}). The proof of Lemma~\ref{lem:cv} (see \cite[Prop.~5.15]{livre-combettes}) tells that $ \sum_{k=0}^{+\infty} \left\| x_{k} - \mathsf{T} (x_{k}) \right\|^2 < \infty $ so
\begin{align*}
\sum_{k=0}^{+\infty} \left\| R_k(R-I)x^0 + R_k d \right\|^2  = \sum_{k=0}^{+\infty} \left\| W \left(  \Lambda_k(\Lambda-I) W^{-1}x^0 + \Lambda_k W^{-1} d \right) \right\|^2 < \infty .
\end{align*}

From the last line, we can deduce that:
\begin{itemize}
\item[i)] the eigenvalues of $R$ are smaller than $1$ in magnitude and $1$ is the only one with this magnitude;
\item[ii)] the algebraic and geometric multiplicities of eigenvalue $1$ coincide as the Jordan form of $R$ does not have block of the form  {\scriptsize
$ J_1 = \left[
\begin{array}{cccc}
1 & 1 & & \\
 & \ddots & \ddots & \\
& & \ddots & 1 \\
&  & & 1
\end{array}
\right] $}. Indeed, if it had one could take  $x^0$ so that the terms in the sum are bounded away from zero (e.g. take $J = [ 1 ~ 1 ;  0 ~ 1]$, then $J_k(J-I) = [ 0 ~ 1 ; 0 ~ 0 ] $).
\end{itemize}

Thus, one can write $R = \left[ W_1 ~ W_2 \vphantom{\begin{array}{c} \underline{W_1}^* \\ \underline{W_2}^* \end{array}} \right] \left[ \begin{array}{cc} I & \\ & \tilde{\Lambda} \end{array} \right] \left[ \begin{array}{c} \underline{W_1}^* \\ \underline{W_2}^* \end{array} \right]$ where:
\begin{itemize}
\item $\tilde{\Lambda}$ is the block diagonal matrix of the Jordan blocks corresponding to the eigenvalues of $R$ with magnitude strictly smaller than $1$;
\item and {\small
$ \hspace*{-0.1cm} \left[ \begin{array}{c} \underline{W_1}^* \\ \underline{W_2}^* \end{array} \right]   \left[ W_1 ~ W_2 \vphantom{\begin{array}{c} \underline{W_1}^* \\ \underline{W_2}^* \end{array}} \right]  = \left[ 
\begin{array}{cc}
\underline{W_1}^* W_1  & \underline{W_1}^* W_2 \\
\underline{W_2}^* W_1  & \underline{W_2}^* W_2
\end{array}
 \right] 
= \left[ 
\begin{array}{cc}
I  & \mathit{0} \\
 \mathit{0} & I
\end{array}
 \right] .
$}
\end{itemize}

From the previous result, $ R = W_1 \underline{W_1}^* + W_2 \tilde{\Lambda} \underline{W_2}^* $ where conveniently $\Pi_{\mathcal N} \triangleq W_1 \underline{W_1}^*$ defines a projection onto ${\mathcal N}$.

Define $\overline{\Pi_{\mathcal N}} \triangleq I - \Pi_{\mathcal N}$ the complementary projection. Let $\bar{x}\in\fix\mathsf{T}$ and define $\Delta_k \triangleq \overline{\Pi_{\mathcal N}}(x_k - \bar{x})$ for all $k>0$. We have
\begin{align}
\nonumber \Delta_{k+1} &\triangleq \overline{\Pi_{\mathcal N}}(x_{k+1} - \bar{x}) =   \overline{\Pi_{\mathcal N}} R (x_k - \bar{x}) \\
 &= W_2 \tilde{\Lambda} \underline{W_2}^*  \underbrace{\overline{\Pi_{\mathcal N}} (x_k - \bar{x}) }_{\Delta_k} = \label{eq:recur} W_2  \tilde{\Lambda}^k \underline{W_2}^*  \Delta^0
\end{align}

Thus $(\Delta_k)_{k>0}$ vanishes exponentially as a consequence of  \cite[Cor.~5.6.14]{HorJoh94} on Eq.~\eqref{eq:recur} which states that there is constant $C\in\mathbb{R}^+$ such that \\ $ \| \Delta_{k} \| \leq C k^n \rho( W_2  \tilde{\Lambda} \underline{W_2}^*)^k $ where the factor $k^n$ stems from the $k$-th power of the Jordan decomposition of $R$ which introduces terms of the form $\nu^k k^\ell$. Using the $\infty$ norm and taking the $\log$ gives the stated result (see \cite[Sec. III-C]{iutzeler2013analysis} or \cite[Chap 3.2.5]{HorJoh94}). Recalling that $ \tilde{\Lambda}$ contains the Jordan blocks associated to the non-unit eigenvalues of $R$, which are all strictly smaller than $1$ in magnitude, $\rho( W_2  \tilde{\Lambda} \underline{W_2}^*)  = \nu < 1 $.  

Finally, we can notice that  $\Pi_{\mathcal N} (x_{k+1}  - x_k) = 0$, and thus  
$$v_{k} = \frac{\|x_{k+1}-x_k\|}{\|x_{k} - x_{k-1}\|} =  \frac{\|\overline{\Pi_{\mathcal N}}(x_{k+1}-x_k)\|}{\|\overline{\Pi_{\mathcal N}}(x_{k} - x_{k-1})\|}  = \frac{\| W_2  \tilde{\Lambda} \underline{W_2}^*  \overline{\Pi_{\mathcal N}}(x_{k}-x_{k-1})\|}{\|\overline{\Pi_{\mathcal N}}(x_{k} - x_{k-1})\|} \leq \|\tilde{\Lambda}  \|  $$
where the inequality tends to be sharper as $k$ grows and $\nu \leq \|\tilde{\Lambda}  \| \leq 1$.

\section{Proof of Lemma~\ref{lem:cvalt}}
\label{apx:cvalt}

Let $\bar{x} \in \fix \mathsf{T}$, and take $k$ even, then $ x_{k+2} = \op{T}\left(  \op{T}(x_k) + \gamma_{k+1}(\op{T}(x_k) - x_k)  \right)$.
\begin{align*}
 & \left\| x_{k+2} -  \bar{x} \right\|^2  =  \left\| \op{T}\left(  \op{T}(x_k) + \gamma_{k+1}(\op{T}(x_k) - x_k)  \right) -  \op{T}(\bar{x}) \right\|^2 \\
&\leq \left\|   \op{T}(x_k)  + \gamma_{k+1}(\op{T}(x_k) - x_k)  -  \bar{x}  \right\|^2  - \frac{1-\alpha}{\alpha} \left\|   \op{T}(x_k) + \gamma_{k+1}(\op{T}(x_k) - x_k)   -  x_{k+2} \right\|^2 \\
&= (1+\gamma_{k+1}) \left\|  \op{T}(x_k)  - \bar{x} \right\|^2  - \gamma_{k+1} \left\|   x_k  -  \bar{x}  \right\|^2 +(1+\gamma_{k+1})\gamma_{k+1} \left\|  \op{T}(x_k)  - x_k \right\|^2 \\
&  \hspace*{0.5cm} - \frac{1-\alpha}{\alpha} \left\|   \op{T}(x_k) + \gamma_{k+1}(\op{T}(x_k) - x_k)   -  x_{k+2} \right\|^2 \\
&\leq (1+\gamma_{k+1}) \left\|  x_k  - \bar{x} \right\|^2   - \gamma_{k+1} \left\|   x_k  -  \bar{x}  \right\|^2 - (1+\gamma_{k+1})\frac{1-\alpha}{\alpha} \left\|  \op{T}(x_k)  - x_{k} \right\|^2 \\
& \hspace*{0.5cm} +(1+\gamma_{k+1})\gamma_{k+1} \left\|  \op{T}(x_k)  - x_k \right\|^2 - \frac{1-\alpha}{\alpha} \left\|   \op{T}(x_k) + \gamma_{k+1}(\op{T}(x_k) - x_k)   -  x_{k+2} \right\|^2 \\
&= \left\|  x_k  - \bar{x} \right\|^2  - (1+\gamma_{k+1})\left( \frac{1-\alpha}{\alpha} -\gamma_{k+1} \right) \left\|  \op{T}(x_k)  - x_{k} \right\|^2 \\
&  \hspace*{0.5cm} - \frac{1-\alpha}{\alpha} \left\|   \op{T}(x_k) + \gamma_{k+1}(\op{T}(x_k) - x_k)   -  x_{k+2} \right\|^2 
\end{align*}
where we used successively: i) the fact that $\op{T}$ is $\alpha$-averaged; ii) the equality of \cite[Cor.~2.14]{livre-combettes}; iii) a second time that $\op{T}$ is $\alpha$-averaged. The assumption on the sequence $(\gamma_k)$ makes the second term negative or null hence it can be dropped.

\noindent We notice that $\| x_{k+2} -  \bar{x} \|^2 \leq \left\|  x_k  - \bar{x} \right\|^2 - \frac{1-\alpha}{\alpha} \left\|   \op{T}(x_k) + \gamma_{k+1}(\op{T}(x_k) - x_k)   -  x_{k+2} \right\|^2 $ implies that the sequence of the \emph{even}  $(\| x_{2k} -  \bar{x} \|^2)_{k>0}$ is decreasing and non-negative, it is thus convergent and the $(x_{2k})_{k>0}$ are bounded. Furthermore,  
$$ \sum_{k=0}^\infty \left\|   \op{T}(x_{2k}) + \gamma_{2k+1}(\op{T}(x_{2k}) - x_{2k})   -  x_{2(k+1)} \right\|^2 < \infty   $$
implies that any limit point of the sequence $(x_{2k})_{k>0}$ belongs to $\fix \op{T}$. 

Let us now take $x^\star$, a limit point of $(x_{2k})_{k>0}$, then $(\| x_{2k} -  {x}^\star \|^2)_{k>0}$ converges and its limit is $\lim_{k\to \infty} \| x_{2k} -  {x}^\star \|^2 = 0$ which means that $x^\star$ is unique. Finally, using non-expansivity, we get that the \emph{odd} sequence also converges to the same point $x^\star$.

\section{Derivation of Inertial ADMM (Sec.~{ref:admmacc})}
\label{apx:inadmm}

The derivations follow nearly the same steps as the one of relaxed ADMM in \cite{eckstein1992douglas} thus we will abridge the common parts. We build upon the ADMM-generating Lions-Mercier operator:
\begin{align*}
 \mathsf{T}_{admm}  = \left\{ \left( \lambda +\rho v , w + \rho v)   \right) : (  u, -{M}^\tran w )\in  \partial f  ;  (v,\lambda)\in\partial g ; w - \rho {M} u = \lambda - \rho v  \right\} .
\end{align*}
but we will consider an \emph{inertial version} of the proximal point algorithm\footnote{we chose to perform the operator \emph{then} the inertia for the sake of clarity and consistency in the derivations.}:
\begin{equation*}
\left\{ \begin{array}{l} x_{k+1} =  \mathsf{T}_{admm} (y_k) \\
y_{k+1} = x_{k+1} + \gamma (x_{k+1} - x_{k} ) \end{array} \right.
\end{equation*}

\noindent {\itshape Representation step}: The input, $y_k$, writes uniquely as $\lambda_k + \rho v_k$ from the representation lemma:
\begin{equation}
\label{eq_opt:inreadmm}
 y_k =  \lambda_k + \rho v_k.
\end{equation}

\noindent {\itshape Mapping step}: The definition of  $  \mathsf{T}_{admm} $ implies that $\lambda_k - \rho v_k$ writes uniquely as $w-\rho {M}u$ with $(u,-{M}^\tran w)\in \partial f$ :
\begin{equation}
\label{eq_opt:inmapadmm}
 w_{k+1} - \rho {M} u_{k+1} = \lambda_k - \rho v_k. 
\end{equation}
Secondly, the output of the resolvent is:
\begin{equation}
\label{eq_opt:inmapadmm2}
x_{k+1} =  w_{k+1} + \rho v_k = \lambda_k + \rho M u_{k+1}.
\end{equation}

\noindent {\itshape Re-representation step}: Here, the proof is a bit different in the \emph{inertial case} as one has to find the values of $\lambda_{k+1}$ and $v_{k+1}$ with $(v_{k+1},\lambda_{k+1}) \in \partial g$, so that $y_{k+1} = x_{k+1} + \gamma( x_{k+1} - x_k )$ writes uniquely as:
\begin{equation}
\label{eq_opt:inrerepadmm}
 y_{k+1} =  \lambda_{k+1} + \rho v_{k+1}.
\end{equation}

Writing Eq.~(\ref{eq_opt:inmapadmm}) of the mapping step, leads to the same step \emph{as for classical ADMM}:
 \begin{eqnarray*}
 & & w_{k+1} - \rho {M} u_{k+1} = \lambda_k - \rho v_k  ~~~ \textrm{ with }  ~~~  (u_{k+1}, -{M}^\tran  w_{k+1}) \in \partial f  \\
&\Rightarrow&  u_{k+1} = \argmin_u \left\{ f (u)  +  \frac{\rho}{2} \left\| {M} u - v_k + \frac{\lambda_k}{\rho} \right\|^2 \right\} .
\end{eqnarray*}

Now, combining Eqs.~(\ref{eq_opt:inmapadmm2}) and  (\ref{eq_opt:inrerepadmm}), we have \emph{(different from classical ADMM)}
 \begin{eqnarray*}
 & &  \lambda_{k+1} + \rho v_{k+1} =  \lambda_k + \rho {M} u_{k+1} + \gamma (  \lambda_k + \rho {M} u_{k+1}  -  (\lambda_{k-1} + \rho {M} u_{k}) )  ~~~ \textrm{ with }  ~~~  (v_{k+1},  \lambda_{k+1}) \in \partial g\\
 &\Rightarrow& \lambda_k + \rho {M} u_{k+1} + \gamma (  \lambda_k + \rho {M} u_{k+1}  -  (\lambda_{k-1} + \rho {M} u_{k}) )  -   \rho v_{k+1} =  \lambda_{k+1} \in  \partial g (v_{k+1})  \\
 &\Rightarrow& 0   \in  \partial g (v_{k+1}) - \rho  \left(  {M} u_{k+1}  -  v_{k+1} + \frac{\lambda_k}{\rho} + \gamma \left( {M} u_{k+1} + \frac{\lambda_k}{\rho}  -  {M} u_{k} - \frac{\lambda_{k-1}}{\rho} \right)    \right) \\
 &\Rightarrow& v_{k+1} = \argmin_v \left\{ g (v)  + \frac{\rho}{2}  \left\|  {M} u_{k+1}  -  v + \frac{\lambda_k}{\rho} + \gamma \left( {M} u_{k+1} + \frac{\lambda_k}{\rho}  -  {M} u_{k} - \frac{\lambda_{k-1}}{\rho} \right)   \right\|^2 \right\} .
\end{eqnarray*}
and the first line also tells us that
$$  \lambda_{k+1} =  \lambda_k +  \rho \left( {M} u_{k+1} - v_{k+1} \right)  + \gamma \left(  \lambda_k + \rho {M} u_{k+1}  -  (\lambda_{k-1} + \rho {M} u_{k}) \right)  $$ 
which we can identify as the iterations of \emph{Inertial ADMM}.

\end{document}

%% file: Figures/alpha_disk.tex
\begin{figure}
\centering
\begin{subfigure}[b]{0.3\columnwidth}
\centering
\setlength\figureheight{0.8\columnwidth} 
\setlength\figurewidth{0.8\columnwidth}
\begin{tikzpicture}[x=\figurewidth,y=\figurewidth][scale=0.5]

 \draw[color=white] (0,0) -- (0,1) ;
 \draw[color=white] (0,1) -- (1,1) ;
 \draw[color=white] (1,1) -- (1,0) ;
 \draw[color=white]  (0,0) -- (1,0) ;

\coordinate (C) at  (0.5,0.5);
\coordinate (U) at  (1,0.5);
\coordinate (I) at  (0.5,1);

\coordinate (A) at  (0.75,0.5);
\coordinate (X) at  (1,0.5);
\coordinate (T) at  (0.73,0.6);

\draw [black,fill=blockbg] (C) circle(0.5);
\draw [black,fill=ablockbg] (A) circle(0.25);

\draw plot[mark size=1.6pt,mark=*,mark options={solid,,,fill=black}] (C);

\draw (0.76,0.33) node[color=black,font=\scriptsize] {FNE};

\draw [line width = 1pt,black,>=stealth,->] (C) -- (U);

\draw [line width = 1pt,black,>=stealth,->] (C) -- (I);

\draw (0.46,0.46) node[color=black,font=\scriptsize] {$0$};
\draw (0.46,1.04) node[color=black,font=\scriptsize] {$i$};
\draw (1.04,0.46) node[color=black,font=\scriptsize] {$1$};

\end{tikzpicture}
\caption{1/2-averaged}
\end{subfigure}
\begin{subfigure}[b]{0.3\columnwidth}
\centering
\setlength\figureheight{0.8\columnwidth} 
\setlength\figurewidth{0.8\columnwidth}
\begin{tikzpicture}[x=\figurewidth,y=\figurewidth]

 \draw[color=white] (0,0) -- (0,1) ;
 \draw[color=white] (0,1) -- (1,1) ;
 \draw[color=white] (1,1) -- (1,0) ;
 \draw[color=white]  (0,0) -- (1,0) ;

\coordinate (C) at  (0.5,0.5);
\coordinate (U) at  (1,0.5);
\coordinate (I) at  (0.5,1);

\coordinate (A) at  (0.666,0.5);
\coordinate (X) at  (1,0.5);

\draw [black,fill=blockbg] (C) circle(0.5);
\draw [black,fill=ablockbg] (A) circle(0.333);

\draw plot[mark size=1.6pt,mark=*,mark options={solid,,,fill=black}] (C);

\draw (0.68,0.33) node[color=black,font=\scriptsize] {$\frac{2}{3}$-averaged};

\draw [line width = 1pt,black,>=stealth,->] (C) -- (U);

\draw [line width = 1pt,black,>=stealth,->] (C) -- (I);

\draw (0.46,0.46) node[color=black,font=\scriptsize] {$0$};
\draw (0.46,1.04) node[color=black,font=\scriptsize] {$i$};
\draw (1.04,0.46) node[color=black,font=\scriptsize] {$1$};

\end{tikzpicture}
\caption{2/3-averaged}
\end{subfigure}
\begin{subfigure}[b]{0.3\columnwidth}
\centering
\setlength\figureheight{0.8\columnwidth} 
\setlength\figurewidth{0.8\columnwidth}

\begin{tikzpicture}[x=\figurewidth,y=\figurewidth]

 \draw[color=white] (0,0) -- (0,1) ;
 \draw[color=white] (0,1) -- (1,1) ;
 \draw[color=white] (1,1) -- (1,0) ;
 \draw[color=white]  (0,0) -- (1,0) ;

\coordinate (C) at  (0.5,0.5);
\coordinate (U) at  (1,0.5);
\coordinate (I) at  (0.5,1);

\coordinate (A) at  (0.75,0.5);
\coordinate (R) at  (0.625,0.5);
\coordinate (X) at  (1,0.5);
\coordinate (T) at  (0.73,0.6);

\draw [black,fill=blockbg] (C) circle(0.5);
\draw [black,fill=tblockbg,opacity=0.6] (R) circle(0.375);
\draw [black,fill=ablockbg,opacity=0.4] (A) circle(0.25);

\draw plot[mark size=1.6pt,mark=*,mark options={solid,,,fill=black}] (C);

\draw (0.76,0.33) node[color=black,font=\scriptsize] {FNE};

\draw (0.38,0.25) node[color=black,font=\scriptsize] {$1.5$ relaxation};

\draw [line width = 1pt,black,>=stealth,->] (C) -- (U);

\draw [line width = 1pt,black,>=stealth,->] (C) -- (I);

\draw (0.46,0.46) node[color=black,font=\scriptsize] {$0$};
\draw (0.46,1.04) node[color=black,font=\scriptsize] {$i$};
\draw (1.04,0.46) node[color=black,font=\scriptsize] {$1$};

\end{tikzpicture}
\caption{$1.5$ relaxation}
\end{subfigure}
\caption{Eigenvalues disks of some $\alpha$-averaged linear operators}
\label{fig:alpha}
\end{figure}

%% file: Figures/comp_grad.tikz
\begin{tikzpicture}[scale=1]

\begin{axis}[
width=\figurewidth,
height=\figureheight,
xmin=0,
xmax=2,
xlabel={Relaxation parameter $\eta$},
xmajorgrids,
xtick={0,1,1.2,2},
xticklabels={$0$,$1$,$\eta^\star$,$2$},
ymin=0,
ymax=1,
ymajorgrids,
ylabel={Rate $\nu$},
ytick={0,0.2,0.33,1},
yticklabels={$0$,$\nu^\star$,$\lambda_{\max}$,$1$},
legend style={at={(0.99,0.99)},anchor=north east,cells={anchor=west},font=\tiny},
]

\addplot [
color=blue!20,
line width=1.4pt,
mark size=0pt
]
coordinates{(0,1) (1.2,0.2) (2,1)};
\addlegendentry{$\nu$};

\addplot [
color=black,
dotted,
line width=1.0pt,
mark size=0pt
]
coordinates{(0,1) (1,0) (2,1)};
\addlegendentry{$|\eta\lambda_{\min} +1-\eta|$};

\addplot [
color=red,
dashed,
line width=0.8pt,
mark size=0pt
] coordinates {(0,1) (1.5,0) (2,0.33)}; 
\addlegendentry{$|\eta\lambda_{\max} +1-\eta|$};

\end{axis} 
\end{tikzpicture}

%% file: Figures/pg_lasso.tikz
\begin{tikzpicture}[scale=0.48]

 \begin{axis}[ 
 width=\columnwidth, 
 height=0.5\columnwidth, 
 xmin=0, 
 xmax=150, 
 xmajorgrids, 
 xlabel={Number of iterations},
 ymode=log,
 ymin=5e-14, 
 ymax=100, 
 yminorticks=true, 
 ymajorgrids, 
 yminorgrids, 
 ylabel={Functional error $F_a(x_k)-F_a^\star$ }, 
 legend columns=2, 
 legend style={at={(0.99,0.99)},anchor=north east,cells={anchor=west}}, 
 ]

 \addplot [ 
 color=blue, 
 solid, 
 line width=1.0pt, 
 mark size=2.5pt, 
 mark=none, 
 mark options={solid,fill=white,draw=blue}, 
 mark repeat={5}, 
 ] 
 file { Figures/data/pg_lasso_orig.dat };
 \addlegendentry{prox. grad.};

 \addplot [ 
 color=cyan, 
 dashed, 
 line width=1.0pt, 
 mark size=1.5pt, 
 mark=*, 
 only marks ,
 mark options={solid,fill=black,draw=cyan}, 
 mark repeat={2}, 
 ] 
  file { Figures/data/pg_lasso_fista.dat };
 \addlegendentry{FISTA};

 \addplot [ 
 color=black,  
 line width=1.0pt, 
 mark size=1.5pt, 
 mark=square*,
 only marks , 
 mark options={solid,fill=white,draw=black}, 
 mark repeat={2}, 
 ] 
   file { Figures/data/pg_lasso_orm.dat };
 \addlegendentry{ORM};

 \addplot [ 
 color=red, 
 dashed, 
 line width=1.0pt, 
 mark size=1.8pt, 
 mark=triangle*,
 only marks , 
 mark options={solid,fill=white,draw=red}, 
 mark repeat={2}, 
 ] 
    file { Figures/data/pg_lasso_oim.dat };
 \addlegendentry{OIM};

 \addplot [ 
 color=magenta, 
 dashed, 
 line width=1.0pt, 
 mark size=1.5pt, 
 mark=*, 
 only marks ,
 mark options={solid,fill=white,draw=magenta}, 
 mark repeat={2}, 
 ] 
    file { Figures/data/pg_lasso_oaim.dat };
 \addlegendentry{OAIM};

  \end{axis} 
 \end{tikzpicture}

%% file: Figures/pg_lassoP.tikz
\begin{tikzpicture} [scale=0.48]

 \begin{axis}[ 
 width=\columnwidth, 
 height=0.5\columnwidth, 
 xmin=0, 
 xmax=150, 
 xmajorgrids, 
 xlabel={Number of iterations},
 ymin=0, 
 ymax=3.9, 
 yminorticks=true, 
 ymajorgrids, 
 yminorgrids, 
 ylabel={Parameters values},
 legend columns=3, 
 legend style={at={(0.99,0.79)},anchor=north east,cells={anchor=west}}, 
 ]

 \addplot [ 
 color=cyan, 
 dashed, 
 line width=1.0pt, 
 mark size=1.5pt, 
 mark=*, 
 only marks ,
 mark options={solid,fill=black,draw=cyan}, 
 mark repeat={2}, 
 ] 
  file { Figures/data/pg_lasso_fistaP.dat };
 \addlegendentry{FISTA};

 \addlegendimage{empty legend} 
 \addlegendentry{}
 
 \addlegendimage{empty legend} 
 \addlegendentry{}

 \addplot [ 
 color=black, 
 dashed, 
 line width=1.0pt, 
 mark size=1.5pt, 
 mark=square*,
 only marks, 
 mark options={solid,fill=white,draw=black}, 
 mark repeat={2}, 
 ] 
   file { Figures/data/pg_lasso_ormP.dat };
 \addlegendentry{ORM};

 \addplot [ 
 color=red, 
 dashed, 
 line width=1.0pt, 
 mark size=1.8pt, 
 mark=triangle*, 
 only marks, 
  mark options={solid,fill=white,draw=red}, 
 mark repeat={2}, 
 ] 
    file { Figures/data/pg_lasso_oimP.dat };
 \addlegendentry{OIM};

 \addplot [ 
 color=magenta, 
 dashed, 
 line width=1.0pt, 
 mark size=1.5pt, 
 mark=*, 
 only marks, 
  mark options={solid,fill=white,draw=magenta}, 
 mark repeat={2}, 
 ] 
     file { Figures/data/pg_lasso_oaimP.dat };
 \addlegendentry{OAIM};

  \end{axis} 
 \end{tikzpicture}

%% file: Figures/pg_l1.tikz
\begin{tikzpicture}[scale=0.48]

 \begin{axis}[ 
 width=\columnwidth, 
 height=0.5\columnwidth, 
 xmin=0, 
 xmax=250, 
 xmajorgrids, 
 xlabel={Number of iterations},
 ymode=log,
 ymin=5e-6, 
 ymax=1, 
 yminorticks=true, 
 ymajorgrids, 
 yminorgrids, 
 ylabel={Functional error $F_b(x_k)-F_b^\star$ }, 
 legend columns=2, 
 legend style={at={(0.99,0.99)},anchor=north east,cells={anchor=west}}, 
 ]

 \addplot [ 
 color=blue, 
 solid, 
 line width=1.0pt, 
 mark size=2.5pt, 
 mark=none, 
 mark options={solid,fill=white,draw=blue}, 
 mark repeat={5}, 
 ] 
 file { Figures/data/pg_l1_orig.dat };
 \addlegendentry{prox. grad.};

 \addplot [ 
 color=cyan, 
 dashed, 
 line width=1.0pt, 
 mark size=1.5pt, 
 mark=*, 
 only marks ,
 mark options={solid,fill=black,draw=cyan}, 
 mark repeat={2}, 
 ] 
  file { Figures/data/pg_l1_fista.dat };
 \addlegendentry{FISTA};

 \addplot [ 
 color=black,  
 line width=1.0pt, 
 mark size=1.5pt, 
 mark=square*,
 only marks , 
 mark options={solid,fill=white,draw=black}, 
 mark repeat={2}, 
 ] 
   file { Figures/data/pg_l1_orm.dat };
 \addlegendentry{ORM};

 \addplot [ 
 color=red, 
 dashed, 
 line width=1.0pt, 
 mark size=1.8pt, 
 mark=triangle*,
 only marks , 
 mark options={solid,fill=white,draw=red}, 
 mark repeat={2}, 
 ] 
    file { Figures/data/pg_l1_oim.dat };
 \addlegendentry{OIM};

 \addplot [ 
 color=magenta, 
 dashed, 
 line width=1.0pt, 
 mark size=1.5pt, 
 mark=*, 
 only marks ,
 mark options={solid,fill=white,draw=magenta}, 
 mark repeat={2}, 
 ] 
    file { Figures/data/pg_l1_oaim.dat };
 \addlegendentry{OAIM};

  \end{axis} 
 \end{tikzpicture}

%% file: Figures/pg_l1P.tikz
\begin{tikzpicture} [scale=0.48]

 \begin{axis}[ 
 width=\columnwidth, 
 height=0.5\columnwidth, 
 xmin=0, 
 xmax=250, 
 xmajorgrids, 
 xlabel={Number of iterations},
 ymin=0, 
 ymax=5.0, 
 yminorticks=true, 
 ymajorgrids, 
 yminorgrids, 
 ylabel={Parameters values},
 legend columns=3, 
 legend style={at={(0.99,0.79)},anchor=north east,cells={anchor=west}}, 
 ]

 \addplot [ 
 color=cyan, 
 dashed, 
 line width=1.0pt, 
 mark size=1.5pt, 
 mark=*, 
 only marks ,
 mark options={solid,fill=black,draw=cyan}, 
 mark repeat={2}, 
 ] 
  file { Figures/data/pg_l1_fistaP.dat };
 \addlegendentry{FISTA};

 \addlegendimage{empty legend} 
 \addlegendentry{}
 
 \addlegendimage{empty legend} 
 \addlegendentry{}

 \addplot [ 
 color=black, 
 dashed, 
 line width=1.0pt, 
 mark size=1.5pt, 
 mark=square*,
 only marks, 
 mark options={solid,fill=white,draw=black}, 
 mark repeat={2}, 
 ] 
   file { Figures/data/pg_l1_ormP.dat };
 \addlegendentry{ORM};

 \addplot [ 
 color=red, 
 dashed, 
 line width=1.0pt, 
 mark size=1.8pt, 
 mark=triangle*, 
 only marks, 
  mark options={solid,fill=white,draw=red}, 
 mark repeat={2}, 
 ] 
    file { Figures/data/pg_l1_oimP.dat };
 \addlegendentry{OIM};

 \addplot [ 
 color=magenta, 
 dashed, 
 line width=1.0pt, 
 mark size=1.5pt, 
 mark=*, 
 only marks, 
  mark options={solid,fill=white,draw=magenta}, 
 mark repeat={2}, 
 ] 
     file { Figures/data/pg_l1_oaimP.dat };
 \addlegendentry{OAIM};

  \end{axis} 
 \end{tikzpicture}

%% file: Figures/pg_l2.tikz
\begin{tikzpicture}[scale=0.48]

 \begin{axis}[ 
 width=\columnwidth, 
 height=0.5\columnwidth, 
 xmin=0, 
 xmax=250, 
 xmajorgrids, 
 xlabel={Number of iterations},
 ymode=log,
 ymin=5e-10, 
 ymax=1, 
 yminorticks=true, 
 ymajorgrids, 
 yminorgrids, 
 ylabel={Functional error $F_c(x_k)-F_c^\star$ }, 
 legend columns=2, 
 legend style={at={(0.99,0.99)},anchor=north east,cells={anchor=west}}, 
 ]

 \addplot [ 
 color=blue, 
 solid, 
 line width=1.0pt, 
 mark size=2.5pt, 
 mark=none, 
 mark options={solid,fill=white,draw=blue}, 
 mark repeat={5}, 
 ] 
 file { Figures/data/pg_l2_orig.dat };
 \addlegendentry{prox. grad.};

 \addplot [ 
 color=cyan, 
 dashed, 
 line width=1.0pt, 
 mark size=1.5pt, 
 mark=*, 
 only marks ,
 mark options={solid,fill=black,draw=cyan}, 
 mark repeat={2}, 
 ] 
  file { Figures/data/pg_l2_fista.dat };
 \addlegendentry{FISTA};

 \addplot [ 
 color=black,  
 line width=1.0pt, 
 mark size=1.5pt, 
 mark=square*,
 only marks , 
 mark options={solid,fill=white,draw=black}, 
 mark repeat={2}, 
 ] 
   file { Figures/data/pg_l2_orm.dat };
 \addlegendentry{ORM};

 \addplot [ 
 color=red, 
 dashed, 
 line width=1.0pt, 
 mark size=1.8pt, 
 mark=triangle*,
 only marks , 
 mark options={solid,fill=white,draw=red}, 
 mark repeat={2}, 
 ] 
    file { Figures/data/pg_l2_oim.dat };
 \addlegendentry{OIM};

 \addplot [ 
 color=magenta, 
 dashed, 
 line width=1.0pt, 
 mark size=1.5pt, 
 mark=*, 
 only marks ,
 mark options={solid,fill=white,draw=magenta}, 
 mark repeat={2}, 
 ] 
    file { Figures/data/pg_l2_oaim.dat };
 \addlegendentry{OAIM};

  \end{axis} 
 \end{tikzpicture}

%% file: Figures/pg_l2P.tikz
\begin{tikzpicture} [scale=0.48]

 \begin{axis}[ 
 width=\columnwidth, 
 height=0.5\columnwidth, 
 xmin=0, 
 xmax=250, 
 xmajorgrids, 
 xlabel={Number of iterations},
 ymin=0, 
 ymax=5.0, 
 yminorticks=true, 
 ymajorgrids, 
 yminorgrids, 
 ylabel={Parameters values},
 legend columns=3, 
 legend style={at={(0.99,0.79)},anchor=north east,cells={anchor=west}}, 
 ]

 \addplot [ 
 color=cyan, 
 dashed, 
 line width=1.0pt, 
 mark size=1.5pt, 
 mark=*, 
 only marks ,
 mark options={solid,fill=black,draw=cyan}, 
 mark repeat={2}, 
 ] 
  file { Figures/data/pg_l2_fistaP.dat };
 \addlegendentry{FISTA};

 \addlegendimage{empty legend} 
 \addlegendentry{}
 
 \addlegendimage{empty legend} 
 \addlegendentry{}

 \addplot [ 
 color=black, 
 dashed, 
 line width=1.0pt, 
 mark size=1.5pt, 
 mark=square*,
 only marks, 
 mark options={solid,fill=white,draw=black}, 
 mark repeat={2}, 
 ] 
   file { Figures/data/pg_l2_ormP.dat };
 \addlegendentry{ORM};

 \addplot [ 
 color=red, 
 dashed, 
 line width=1.0pt, 
 mark size=1.8pt, 
 mark=triangle*, 
 only marks, 
  mark options={solid,fill=white,draw=red}, 
 mark repeat={2}, 
 ] 
    file { Figures/data/pg_l2_oimP.dat };
 \addlegendentry{OIM};

 \addplot [ 
 color=magenta, 
 dashed, 
 line width=1.0pt, 
 mark size=1.5pt, 
 mark=*, 
 only marks, 
  mark options={solid,fill=white,draw=magenta}, 
 mark repeat={2}, 
 ] 
     file { Figures/data/pg_l2_oaimP.dat };
 \addlegendentry{OAIM};

  \end{axis} 
 \end{tikzpicture}

%% file: Figures/admm_lasso.tikz
\begin{tikzpicture}[scale=0.48]

 \begin{axis}[ 
 width=\columnwidth, 
 height=0.5\columnwidth, 
 xmin=0, 
 xmax=50, 
  xlabel={Number of iterations},
 xmajorgrids, 
 ymode=log,
 ymin=5e-16, 
 ymax=2, 
 yminorticks=true, 
 ymajorgrids, 
 yminorgrids, 
 ylabel={Functional error }, 
 legend columns=2, 
 legend style={at={(0.99,0.99)},anchor=north east,cells={anchor=west}}, 
 ]

 \addplot [ 
 color=blue, 
 solid, 
 line width=1.0pt, 
 mark size=2.5pt, 
 mark=none, 
 mark options={solid,fill=white,draw=blue}, 
 mark repeat={5}, 
 ] 
  file { Figures/data/admm_lasso_orig.dat };
 \addlegendentry{ADMM};

 \addplot [ 
 color=black,  
 line width=1.0pt, 
 mark size=1.5pt, 
 mark=square*,
 only marks , 
 mark options={solid,fill=white,draw=black}, 
 mark repeat={1}, 
 ] 
  file { Figures/data/admm_lasso_orm.dat };
 \addlegendentry{ORM};

 \addplot [ 
 color=red, 
 dashed, 
 line width=1.0pt, 
 mark size=1.8pt, 
 mark=triangle*,
 only marks , 
 mark options={solid,fill=white,draw=red}, 
 mark repeat={1}, 
 ] 
  file { Figures/data/admm_lasso_oim.dat };
 \addlegendentry{OIM};

 \addplot [ 
 color=magenta, 
 dashed, 
 line width=1.0pt, 
 mark size=1.5pt, 
 mark=*, 
 only marks ,
 mark options={solid,fill=white,draw=magenta}, 
 mark repeat={1}, 
 ] 
  file { Figures/data/admm_lasso_oaim.dat };
 \addlegendentry{OAIM};

 \addplot [ 
 color=cyan, 
 dashed, 
 line width=1.0pt, 
 mark size=1.5pt, 
 mark=*, 
 only marks ,
 mark options={solid,fill=black,draw=cyan}, 
 mark repeat={1}, 
 ] 
   file { Figures/data/admm_lasso_fista.dat };
 \addlegendentry{Fast ADMM};

  \end{axis} 
 \end{tikzpicture}

%% file: Figures/admm_lassoP.tikz
\begin{tikzpicture}[scale=0.48]

 \begin{axis}[ 
 width=\columnwidth, 
 height=0.5\columnwidth, 
 xmin=0, 
 xmax=50, 
 xmajorgrids, 
 xlabel={Number of iterations},
 ymin=0, 
 ymax=2.8, 
 yminorticks=true, 
 ymajorgrids, 
 yminorgrids, 
 ylabel={Parameters values},
 legend columns=3, 
 legend style={at={(0.99,0.9)},anchor=north east,cells={anchor=west}}, 
 ]

 \addplot [ 
 color=black, 
 dashed, 
 line width=1.0pt, 
 mark size=1.5pt, 
 mark=square*,
 only marks, 
 mark options={solid,fill=white,draw=black}, 
 mark repeat={1}, 
 ] 
   file { Figures/data/admm_lasso_ormP.dat };
 \addlegendentry{ORM};

 \addplot [ 
 color=red, 
 dashed, 
 line width=1.0pt, 
 mark size=1.8pt, 
 mark=triangle*, 
 only marks, 
  mark options={solid,fill=white,draw=red}, 
 mark repeat={1}, 
 ] 
   file { Figures/data/admm_lasso_oimP.dat };
 \addlegendentry{OIM};

 \addplot [ 
 color=magenta, 
 dashed, 
 line width=1.0pt, 
 mark size=1.5pt, 
 mark=*, 
 only marks, 
  mark options={solid,fill=white,draw=magenta}, 
 mark repeat={1}, 
 ] 
  file { Figures/data/admm_lasso_oaimP.dat };
 \addlegendentry{OAIM};

 \addplot [ 
 color=cyan, 
 dashed, 
 line width=1.0pt, 
 mark size=1.5pt, 
 mark=*, 
 only marks ,
 mark options={solid,fill=black,draw=cyan}, 
 mark repeat={1}, 
 ] 
  file { Figures/data/admm_lasso_fistaP.dat };
 \addlegendentry{Fast ADMM};

  \end{axis} 
 \end{tikzpicture}

%% file: Figures/admm_l1.tikz
\begin{tikzpicture}[scale=0.48]

 \begin{axis}[ 
 width=\columnwidth, 
 height=0.5\columnwidth, 
 xmin=0, 
 xmax=40, 
  xlabel={Number of iterations},
 xmajorgrids, 
 ymode=log,
 ymin=5e-11, 
 ymax=2, 
 yminorticks=true, 
 ymajorgrids, 
 yminorgrids, 
 ylabel={Functional error }, 
 legend columns=2, 
 legend style={at={(0.99,0.99)},anchor=north east,cells={anchor=west}}, 
 ]

 \addplot [ 
 color=blue, 
 solid, 
 line width=1.0pt, 
 mark size=2.5pt, 
 mark=none, 
 mark options={solid,fill=white,draw=blue}, 
 mark repeat={5}, 
 ] 
  file { Figures/data/admm_l1_orig.dat };
 \addlegendentry{ADMM};

 \addplot [ 
 color=black,  
 line width=1.0pt, 
 mark size=1.5pt, 
 mark=square*,
 only marks , 
 mark options={solid,fill=white,draw=black}, 
 mark repeat={1}, 
 ] 
  file { Figures/data/admm_l1_orm.dat };
 \addlegendentry{ORM};

 \addplot [ 
 color=red, 
 dashed, 
 line width=1.0pt, 
 mark size=1.8pt, 
 mark=triangle*,
 only marks , 
 mark options={solid,fill=white,draw=red}, 
 mark repeat={1}, 
 ] 
  file { Figures/data/admm_l1_oim.dat };
 \addlegendentry{OIM};

 \addplot [ 
 color=magenta, 
 dashed, 
 line width=1.0pt, 
 mark size=1.5pt, 
 mark=*, 
 only marks ,
 mark options={solid,fill=white,draw=magenta}, 
 mark repeat={1}, 
 ] 
  file { Figures/data/admm_l1_oaim.dat };
 \addlegendentry{OAIM};

 \addplot [ 
 color=cyan, 
 dashed, 
 line width=1.0pt, 
 mark size=1.5pt, 
 mark=*, 
 only marks ,
 mark options={solid,fill=black,draw=cyan}, 
 mark repeat={1}, 
 ] 
   file { Figures/data/admm_l1_fista.dat };
 \addlegendentry{Fast ADMM};

  \end{axis} 
 \end{tikzpicture}

%% file: Figures/admm_l1P.tikz
\begin{tikzpicture}[scale=0.48]

 \begin{axis}[ 
 width=\columnwidth, 
 height=0.5\columnwidth, 
 xmin=0, 
 xmax=40, 
 xmajorgrids, 
 xlabel={Number of iterations},
 ymin=0, 
 ymax=2.5, 
 yminorticks=true, 
 ymajorgrids, 
 yminorgrids, 
 ylabel={Parameters values},
 legend columns=3, 
 legend style={at={(0.99,0.6)},anchor=north east,cells={anchor=west}}, 
 ]

 \addplot [ 
 color=black, 
 dashed, 
 line width=1.0pt, 
 mark size=1.5pt, 
 mark=square*,
 only marks, 
 mark options={solid,fill=white,draw=black}, 
 mark repeat={1}, 
 ] 
   file { Figures/data/admm_l1_ormP.dat };
 \addlegendentry{ORM};

 \addplot [ 
 color=red, 
 dashed, 
 line width=1.0pt, 
 mark size=1.8pt, 
 mark=triangle*, 
 only marks, 
  mark options={solid,fill=white,draw=red}, 
 mark repeat={1}, 
 ] 
   file { Figures/data/admm_l1_oimP.dat };
 \addlegendentry{OIM};

 \addplot [ 
 color=magenta, 
 dashed, 
 line width=1.0pt, 
 mark size=1.5pt, 
 mark=*, 
 only marks, 
  mark options={solid,fill=white,draw=magenta}, 
 mark repeat={1}, 
 ] 
  file { Figures/data/admm_l1_oaimP.dat };
 \addlegendentry{OAIM};

 \addplot [ 
 color=cyan, 
 dashed, 
 line width=1.0pt, 
 mark size=1.5pt, 
 mark=*, 
 only marks ,
 mark options={solid,fill=black,draw=cyan}, 
 mark repeat={1}, 
 ] 
  file { Figures/data/admm_l1_fistaP.dat };
 \addlegendentry{Fast ADMM};

  \end{axis} 
 \end{tikzpicture}

%% file: Figures/admm_l2.tikz
\begin{tikzpicture}[scale=0.48]

 \begin{axis}[ 
 width=\columnwidth, 
 height=0.5\columnwidth, 
 xmin=0, 
 xmax=125, 
  xlabel={Number of iterations},
 xmajorgrids, 
 ymode=log,
 ymin=5e-11, 
 ymax=2, 
 yminorticks=true, 
 ymajorgrids, 
 yminorgrids, 
 ylabel={Functional error }, 
 legend columns=2, 
 legend style={at={(0.99,0.99)},anchor=north east,cells={anchor=west}}, 
 ]

 \addplot [ 
 color=blue, 
 solid, 
 line width=1.0pt, 
 mark size=2.5pt, 
 mark=none, 
 mark options={solid,fill=white,draw=blue}, 
 mark repeat={5}, 
 ] 
  file { Figures/data/admm_l2_orig.dat };
 \addlegendentry{ADMM};

 \addplot [ 
 color=black,  
 line width=1.0pt, 
 mark size=1.5pt, 
 mark=square*,
 only marks , 
 mark options={solid,fill=white,draw=black}, 
 mark repeat={1}, 
 ] 
  file { Figures/data/admm_l2_orm.dat };
 \addlegendentry{ORM};

 \addplot [ 
 color=red, 
 dashed, 
 line width=1.0pt, 
 mark size=1.8pt, 
 mark=triangle*,
 only marks , 
 mark options={solid,fill=white,draw=red}, 
 mark repeat={1}, 
 ] 
  file { Figures/data/admm_l2_oim.dat };
 \addlegendentry{OIM};

 \addplot [ 
 color=magenta, 
 dashed, 
 line width=1.0pt, 
 mark size=1.5pt, 
 mark=*, 
 only marks ,
 mark options={solid,fill=white,draw=magenta}, 
 mark repeat={1}, 
 ] 
  file { Figures/data/admm_l2_oaim.dat };
 \addlegendentry{OAIM};

 \addplot [ 
 color=cyan, 
 dashed, 
 line width=1.0pt, 
 mark size=1.5pt, 
 mark=*, 
 only marks ,
 mark options={solid,fill=black,draw=cyan}, 
 mark repeat={1}, 
 ] 
   file { Figures/data/admm_l2_fista.dat };
 \addlegendentry{Fast ADMM};

  \end{axis} 
 \end{tikzpicture}

%% file: Figures/admm_l2P.tikz
\begin{tikzpicture}[scale=0.48]

 \begin{axis}[ 
 width=\columnwidth, 
 height=0.5\columnwidth, 
 xmin=0, 
 xmax=125, 
 xmajorgrids, 
 xlabel={Number of iterations},
 ymin=0, 
 ymax=5.0, 
 yminorticks=true, 
 ymajorgrids, 
 yminorgrids, 
 ylabel={Parameters values},
 legend columns=3, 
 legend style={at={(0.99,0.8)},anchor=north east,cells={anchor=west}}, 
 ]

 \addplot [ 
 color=black, 
 dashed, 
 line width=1.0pt, 
 mark size=1.5pt, 
 mark=square*,
 only marks, 
 mark options={solid,fill=white,draw=black}, 
 mark repeat={1}, 
 ] 
   file { Figures/data/admm_l2_ormP.dat };
 \addlegendentry{ORM};

 \addplot [ 
 color=red, 
 dashed, 
 line width=1.0pt, 
 mark size=1.8pt, 
 mark=triangle*, 
 only marks, 
  mark options={solid,fill=white,draw=red}, 
 mark repeat={1}, 
 ] 
   file { Figures/data/admm_l2_oimP.dat };
 \addlegendentry{OIM};

 \addplot [ 
 color=magenta, 
 dashed, 
 line width=1.0pt, 
 mark size=1.5pt, 
 mark=*, 
 only marks, 
  mark options={solid,fill=white,draw=magenta}, 
 mark repeat={1}, 
 ] 
  file { Figures/data/admm_l2_oaimP.dat };
 \addlegendentry{OAIM};

 \addplot [ 
 color=cyan, 
 dashed, 
 line width=1.0pt, 
 mark size=1.5pt, 
 mark=*, 
 only marks ,
 mark options={solid,fill=black,draw=cyan}, 
 mark repeat={1}, 
 ] 
  file { Figures/data/admm_l2_fistaP.dat };
 \addlegendentry{Fast ADMM};

  \end{axis} 
 \end{tikzpicture}

%% file: Figures/pd.tikz
\begin{tikzpicture} [scale=0.48]

 \begin{axis}[ 
 width=\columnwidth, 
 height=0.5\columnwidth, 
 xmin=0, 
 xmax=100, 
 xmajorgrids, 
 xlabel={Number of iterations},
 ymode=log,
 ymin=5e-12, 
 ymax=2, 
 yminorticks=true, 
 ymajorgrids, 
 yminorgrids, 
 ylabel={Functional error }, 
 legend columns=2, 
 legend style={at={(0.99,0.99)},anchor=north east,cells={anchor=west}}, 
 ]

 \addplot [ 
 color=blue, 
 solid, 
 line width=1.0pt, 
 mark size=2.5pt, 
 mark=none, 
 mark options={solid,fill=white,draw=blue}, 
 mark repeat={5}, 
 ] 
 coordinates{ 
( 0 , 12.67274222090118129813 )
( 1 , 94.26014021420223798486 )
( 2 , 49.42634201172482733000 )
( 3 , 21.28951456871344305455 )
( 4 , 9.82873418787273678277 )
( 5 , 5.33162933463453114769 )
( 6 , 3.36613061162805138338 )
( 7 , 2.35575377277972464185 )
( 8 , 1.74712422142233947397 )
( 9 , 1.34498261892262505057 )
( 10 , 1.06123599617072272849 )
( 11 , 0.84936382589288683675 )
( 12 , 0.68787757922751424644 )
( 13 , 0.56082276545860310080 )
( 14 , 0.45963965916315885352 )
( 15 , 0.37891525064403808187 )
( 16 , 0.31350074574987019105 )
( 17 , 0.26121384547332482384 )
( 18 , 0.21794839191518633470 )
( 19 , 0.18236914520070257595 )
( 20 , 0.15350027837827440180 )
( 21 , 0.12968695736347513048 )
( 22 , 0.10987762986794180620 )
( 23 , 0.09322992335615509774 )
( 24 , 0.07915909456567504776 )
( 25 , 0.06720622639371498508 )
( 26 , 0.05697262951253279084 )
( 27 , 0.04835240776932536733 )
( 28 , 0.04103851383536749609 )
( 29 , 0.03478876179020318204 )
( 30 , 0.02944650960564132447 )
( 31 , 0.02488575690067662549 )
( 32 , 0.02103099591455048767 )
( 33 , 0.01781586773207521901 )
( 34 , 0.01509799785058518751 )
( 35 , 0.01278311410940702331 )
( 36 , 0.01081700348041536586 )
( 37 , 0.00915167639358926976 )
( 38 , 0.00774873930182451431 )
( 39 , 0.00656199535712076454 )
( 40 , 0.00556008959505405187 )
( 41 , 0.00471636450692969333 )
( 42 , 0.00399907467023297158 )
( 43 , 0.00338898957001987355 )
( 44 , 0.00286983558185838206 )
( 45 , 0.00242962556498937943 )
( 46 , 0.00206258705110329288 )
( 47 , 0.00175377885594585337 )
( 48 , 0.00149340396763975036 )
( 49 , 0.00127302828380848609 )
( 50 , 0.00108615353862262509 )
( 51 , 0.00092743164567821168 )
( 52 , 0.00079242083070951708 )
( 53 , 0.00067742001462534063 )
( 54 , 0.00057933805391385818 )
( 55 , 0.00049558693551432498 )
( 56 , 0.00042399449782948295 )
( 57 , 0.00036294971601336101 )
( 58 , 0.00031119544598467996 )
( 59 , 0.00026707659748126389 )
( 60 , 0.00022940531427551036 )
( 61 , 0.00019720259554034669 )
( 62 , 0.00016964746785674834 )
( 63 , 0.00014604760097469693 )
( 64 , 0.00012581768612918154 )
( 65 , 0.00010846208212811348 )
( 66 , 0.00009356051425157830 )
( 67 , 0.00008075617786573730 )
( 68 , 0.00006974580165319821 )
( 69 , 0.00006027132886643471 )
( 70 , 0.00005211294115525789 )
( 71 , 0.00004508320009577460 )
( 72 , 0.00003902212009165851 )
( 73 , 0.00003379301879391505 )
( 74 , 0.00002927901741855976 )
( 75 , 0.00002538008507713130 )
( 76 , 0.00002201053945682929 )
( 77 , 0.00001909693017765335 )
( 78 , 0.00001657624457962470 )
( 79 , 0.00001439438461048326 )
( 80 , 0.00001250487284210067 )
( 81 , 0.00001086775151648567 )
( 82 , 0.00000944864526175593 )
( 83 , 0.00000821796228578364 )
( 84 , 0.00000715021273833827 )
( 85 , 0.00000622342702882861 )
( 86 , 0.00000541865843217693 )
( 87 , 0.00000471955785386058 )
( 88 , 0.00000411200997518790 )
( 89 , 0.00000358382135345892 )
( 90 , 0.00000312445328276567 )
( 91 , 0.00000272479258356384 )
( 92 , 0.00000237695478944033 )
( 93 , 0.00000207411519426159 )
( 94 , 0.00000181036360302755 )
( 95 , 0.00000158057939003697 )
( 96 , 0.00000138032409680022 )
( 97 , 0.00000120574876660839 )
( 98 , 0.00000105351424650735 )
( 99 , 0.00000092072245649888 )
( 100 , 0.00000080485712672385 )
};
 \addlegendentry{Primal-Dual};

 \addplot [ 
 color=black,  
 line width=1.0pt, 
 mark size=1.5pt, 
 mark=square*,
 only marks , 
 mark options={solid,fill=white,draw=black}, 
 mark repeat={2}, 
 ] 
 coordinates{ 
( 0 , 12.67274222090118129813 )
( 1 , 94.26014021420223798486 )
( 2 , 49.42634201172482733000 )
( 3 , 21.28951456871344305455 )
( 4 , 9.82873418787273678277 )
( 5 , 5.33162933463453114769 )
( 6 , 3.36613061162805138338 )
( 7 , 2.35575377277972464185 )
( 8 , 1.74712422142233947397 )
( 9 , 1.34498261892262505057 )
( 10 , 1.06123599617072272849 )
( 11 , 0.75232505428084195387 )
( 12 , 0.52669146947203415721 )
( 13 , 0.37801355295421501523 )
( 14 , 0.27723282136028615241 )
( 15 , 0.20562416719389986497 )
( 16 , 0.15156711967747682479 )
( 17 , 0.11384999512564419888 )
( 18 , 0.08528862871503406495 )
( 19 , 0.06341087480546292454 )
( 20 , 0.04835994329374670997 )
( 21 , 0.03612852976898395241 )
( 22 , 0.02851029229791279818 )
( 23 , 0.02128350677347157216 )
( 24 , 0.01610150129886278592 )
( 25 , 0.01249358788102483686 )
( 26 , 0.00913830869297527215 )
( 27 , 0.00729223733223349768 )
( 28 , 0.00511932922158209180 )
( 29 , 0.00413300540630245905 )
( 30 , 0.00314146204764398362 )
( 31 , 0.00242734601741645406 )
( 32 , 0.00179055941801209428 )
( 33 , 0.00138790580086833870 )
( 34 , 0.00102425110450710122 )
( 35 , 0.00079847442286506976 )
( 36 , 0.00061302194549384126 )
( 37 , 0.00047796862910942650 )
( 38 , 0.00035747685572218302 )
( 39 , 0.00027992982289859469 )
( 40 , 0.00021000059507159108 )
( 41 , 0.00016526398633587291 )
( 42 , 0.00012437242018492611 )
( 43 , 0.00009830599772442383 )
( 44 , 0.00007418969394024089 )
( 45 , 0.00005886840558133599 )
( 46 , 0.00004453437787788062 )
( 47 , 0.00003545954567130138 )
( 48 , 0.00002688206033951701 )
( 49 , 0.00002147082969372605 )
( 50 , 0.00001630741117963908 )
( 51 , 0.00001306144937984755 )
( 52 , 0.00000993668005833115 )
( 53 , 0.00000797913133609995 )
( 54 , 0.00000607907856320367 )
( 55 , 0.00000489283397442364 )
( 56 , 0.00000373251579688372 )
( 57 , 0.00000301053384177408 )
( 58 , 0.00000229920287253549 )
( 59 , 0.00000185805161834196 )
( 60 , 0.00000142044209638925 )
( 61 , 0.00000114992592159524 )
( 62 , 0.00000087985717556194 )
( 63 , 0.00000071344190999412 )
( 64 , 0.00000054629462198363 )
( 65 , 0.00000044362310447355 )
( 66 , 0.00000033990878378631 )
( 67 , 0.00000027639934430113 )
( 68 , 0.00000021189632448682 )
( 69 , 0.00000017251920070294 )
( 70 , 0.00000013231962725513 )
( 71 , 0.00000010785362292154 )
( 72 , 0.00000008275382512579 )
( 73 , 0.00000006752373593599 )
( 74 , 0.00000005182577211826 )
( 75 , 0.00000004232900963075 )
( 76 , 0.00000003249648017345 )
( 77 , 0.00000002656573983018 )
( 78 , 0.00000002039880087068 )
( 79 , 0.00000001668999516369 )
( 80 , 0.00000001281748751580 )
( 81 , 0.00000001049531661579 )
( 82 , 0.00000000806095101780 )
( 83 , 0.00000000660540777631 )
( 84 , 0.00000000507361974655 )
( 85 , 0.00000000416036982642 )
( 86 , 0.00000000319569437579 )
( 87 , 0.00000000262219401748 )
( 88 , 0.00000000201416128220 )
( 89 , 0.00000000165372426864 )
( 90 , 0.00000000127022659058 )
( 91 , 0.00000000104354214159 )
( 92 , 0.00000000080149575865 )
( 93 , 0.00000000065882943545 )
( 94 , 0.00000000050598458756 )
( 95 , 0.00000000041614356405 )
( 96 , 0.00000000031958435898 )
( 97 , 0.00000000026296120836 )
( 98 , 0.00000000020190427108 )
( 99 , 0.00000000016622792032 )
( 100 , 0.00000000012762413348 )
};
 \addlegendentry{ORM};

 \addplot [ 
 color=red, 
 dashed, 
 line width=1.0pt, 
 mark size=1.8pt, 
 mark=triangle*,
 only marks , 
 mark options={solid,fill=white,draw=red}, 
 mark repeat={2}, 
 ] 
 coordinates{ 
( 0 , 12.67274222090118129813 )
( 1 , 94.26014021420223798486 )
( 2 , 49.42634201172482733000 )
( 3 , 21.28951456871344305455 )
( 4 , 9.82873418787273678277 )
( 5 , 5.33162933463453114769 )
( 6 , 3.36613061162805138338 )
( 7 , 2.35575377277972464185 )
( 8 , 1.74712422142233947397 )
( 9 , 1.34498261892262505057 )
( 10 , 1.06123599617072272849 )
( 11 , 0.84936382589288683675 )
( 12 , 0.62950884015766916946 )
( 13 , 0.45395383621691109965 )
( 14 , 0.26883450339214931546 )
( 15 , 0.13508309842987031857 )
( 16 , 0.06999305121263077467 )
( 17 , 0.06099943269280672098 )
( 18 , 0.11165043354453274560 )
( 19 , 0.09375066738043003056 )
( 20 , 0.07275441676351235287 )
( 21 , 0.05424972705737474143 )
( 22 , 0.03364178399703732225 )
( 23 , 0.01627230177936311861 )
( 24 , 0.00663947377573492759 )
( 25 , 0.00437644203223541695 )
( 26 , 0.01367401707923932008 )
( 27 , 0.01156380302213833033 )
( 28 , 0.00901134960975014110 )
( 29 , 0.00670336265479321014 )
( 30 , 0.00409450404281486158 )
( 31 , 0.00189228485388781564 )
( 32 , 0.00060199505543678811 )
( 33 , 0.00030139202813828092 )
( 34 , 0.00073326615121871441 )
( 35 , 0.00151407689917348875 )
( 36 , 0.00020619744529781769 )
( 37 , 0.00014857462048922798 )
( 38 , 0.00009369466074105048 )
( 39 , 0.00005694203590778102 )
( 40 , 0.00002881451693070858 )
( 41 , 0.00001487217036455490 )
( 42 , 0.00001192771956226579 )
( 43 , 0.00001455473153200160 )
( 44 , 0.00001120911094076860 )
( 45 , 0.00000884232692044407 )
( 46 , 0.00000650285264569561 )
( 47 , 0.00000476149985928487 )
( 48 , 0.00000306280470496745 )
( 49 , 0.00000170932747067809 )
( 50 , 0.00000080318432438276 )
( 51 , 0.00000038964086002125 )
( 52 , 0.00000043573250607665 )
( 53 , 0.00000081822214426097 )
( 54 , 0.00000024417020938472 )
( 55 , 0.00000016021570914404 )
( 56 , 0.00000009157938762883 )
( 57 , 0.00000005540176672980 )
( 58 , 0.00000003529280689918 )
( 59 , 0.00000002480833316554 )
( 60 , 0.00000004462250302595 )
( 61 , 0.00000003686652760848 )
( 62 , 0.00000003093997946735 )
( 63 , 0.00000002623290740189 )
( 64 , 0.00000002086163775061 )
( 65 , 0.00000001620629674903 )
( 66 , 0.00000001074922550970 )
( 67 , 0.00000000585490766980 )
( 68 , 0.00000000239301911620 )
( 69 , 0.00000000071770500654 )
( 70 , 0.00000000073004002843 )
( 71 , 0.00000000196432736743 )
( 72 , 0.00000000056592242004 )
( 73 , 0.00000000045966430662 )
( 74 , 0.00000000034483349509 )
( 75 , 0.00000000025222135491 )
( 76 , 0.00000000015800694086 )
( 77 , 0.00000000008353850944 )
( 78 , 0.00000000003758771072 )
( 79 , 0.00000000001988453846 )
( 80 , 0.00000000002517452913 )
( 81 , 0.00000000004599343129 )
( 82 , 0.00000000001367439495 )
( 83 , 0.00000000000979127890 )
( 84 , 0.00000000000623856522 )
( 85 , 0.00000000000395061761 )
( 86 , 0.00000000000215649720 )
( 87 , 0.00000000000114042109 )
( 88 , 0.00000000000319388960 )
( 89 , 0.00000000000260413913 )
( 90 , 0.00000000000214939178 )
( 91 , 0.00000000000176925141 )
( 92 , 0.00000000000135713663 )
( 93 , 0.00000000000100541797 )
( 94 , 0.00000000000062172489 )
( 95 , 0.00000000000031263880 )
( 96 , 0.00000000000011368684 )
( 97 , 0.00000000000006750156 )
( 98 , 0.00000000000025579538 )
( 99 , 0.00000000000022026825 )
( 100 , 0.00000000000019184654 )
};
 \addlegendentry{OIM};

 \addplot [ 
 color=magenta, 
 dashed, 
 line width=1.0pt, 
 mark size=1.5pt, 
 mark=*, 
 only marks ,
 mark options={solid,fill=white,draw=magenta}, 
 mark repeat={2}, 
 ] 
 coordinates{ 
( 0 , 12.67274222090118129813 )
( 1 , 94.26014021420223798486 )
( 2 , 49.42634201172482733000 )
( 3 , 21.28951456871344305455 )
( 4 , 9.82873418787273678277 )
( 5 , 5.33162933463453114769 )
( 6 , 3.36613061162805138338 )
( 7 , 2.35575377277972464185 )
( 8 , 1.74712422142233947397 )
( 9 , 1.34498261892262505057 )
( 10 , 1.06123599617072272849 )
( 11 , 0.84936382589288683675 )
( 12 , 0.68787757922751424644 )
( 13 , 0.56082276545860310080 )
( 14 , 0.26766563678046395580 )
( 15 , 0.22239413494620308143 )
( 16 , 0.11554125914611290682 )
( 17 , 0.09727540173001614221 )
( 18 , 0.04794531453784145469 )
( 19 , 0.04034016622586733547 )
( 20 , 0.01992462203090994421 )
( 21 , 0.01663540915018302258 )
( 22 , 0.00788777490338432585 )
( 23 , 0.00647921587888333761 )
( 24 , 0.00314410531244746494 )
( 25 , 0.00251978393251306443 )
( 26 , 0.00131001425284082984 )
( 27 , 0.00101335940192370799 )
( 28 , 0.00058122101163249340 )
( 29 , 0.00042604641722832071 )
( 30 , 0.00028246783089613814 )
( 31 , 0.00018858271036137353 )
( 32 , 0.00015804776164429768 )
( 33 , 0.00008992408804431307 )
( 34 , 0.00010991851378250317 )
( 35 , 0.00004960463922998315 )
( 36 , 0.00008791329908675038 )
( 37 , 0.00003275176710459959 )
( 38 , 0.00009704801340859603 )
( 39 , 0.00003104478077631256 )
( 40 , 0.00011297916618602244 )
( 41 , 0.00003399797679648486 )
( 42 , 0.00001678015397743593 )
( 43 , 0.00001097429364804725 )
( 44 , 0.00000834637483748679 )
( 45 , 0.00000686348553458060 )
( 46 , 0.00000528763527540832 )
( 47 , 0.00000457581079160718 )
( 48 , 0.00000361931697057116 )
( 49 , 0.00000315813033324730 )
( 50 , 0.00000168051973403749 )
( 51 , 0.00000146775465736937 )
( 52 , 0.00000078786238688622 )
( 53 , 0.00000068827302612817 )
( 54 , 0.00000035588483271454 )
( 55 , 0.00000031048625004360 )
( 56 , 0.00000016238356792542 )
( 57 , 0.00000014121297198244 )
( 58 , 0.00000007456582196141 )
( 59 , 0.00000006445129230315 )
( 60 , 0.00000003467428655313 )
( 61 , 0.00000002968199552811 )
( 62 , 0.00000001632533752627 )
( 63 , 0.00000001376144354026 )
( 64 , 0.00000000787197862451 )
( 65 , 0.00000000646903330903 )
( 66 , 0.00000000397012911435 )
( 67 , 0.00000000308275360794 )
( 68 , 0.00000000223081642048 )
( 69 , 0.00000000151872470155 )
( 70 , 0.00000000155009516334 )
( 71 , 0.00000000081144690967 )
( 72 , 0.00000000135809230528 )
( 73 , 0.00000000052182613786 )
( 74 , 0.00000000121619692095 )
( 75 , 0.00000000101826458376 )
( 76 , 0.00000000087041840402 )
( 77 , 0.00000000075531758625 )
( 78 , 0.00000000060024163417 )
( 79 , 0.00000000052775916970 )
( 80 , 0.00000000042179948423 )
( 81 , 0.00000000037113068174 )
( 82 , 0.00000000019607782065 )
( 83 , 0.00000000017251267082 )
( 84 , 0.00000000009125145084 )
( 85 , 0.00000000008027356557 )
( 86 , 0.00000000004201083925 )
( 87 , 0.00000000003695177497 )
( 88 , 0.00000000001936584226 )
( 89 , 0.00000000001702105124 )
( 90 , 0.00000000000893152219 )
( 91 , 0.00000000000784439180 )
( 92 , 0.00000000000413535872 )
( 93 , 0.00000000000360600438 )
( 94 , 0.00000000000192201810 )
( 95 , 0.00000000000166622272 )
( 96 , 0.00000000000091304742 )
( 97 , 0.00000000000076383344 )
( 98 , 0.00000000000044764192 )
( 99 , 0.00000000000036237680 )
( 100 , 0.00000000000021316282 )
};
 \addlegendentry{OAIM};

  \end{axis} 
 \end{tikzpicture}

%% file: Figures/pd_step.tikz
\begin{tikzpicture} [scale=0.48] 

 \begin{axis}[ 
 width=\columnwidth, 
 height=0.5\columnwidth, 
 xmin=0, 
 xmax=100, 
 xmajorgrids, 
 xlabel={Number of iterations},
 ymin=0, 
 ymax=4, 
 yminorticks=true, 
 ymajorgrids, 
 yminorgrids, 
 ylabel={Parameters values},
 legend columns=3, 
 legend style={at={(0.99,0.99)},anchor=north east,cells={anchor=west}}, 
 ]

 \addplot [ 
 color=black, 
 dashed, 
 line width=1.0pt, 
 mark size=1.5pt, 
 mark=square*,
 only marks, 
 mark options={solid,fill=white,draw=black}, 
 mark repeat={2}, 
 ] 
 coordinates{ 
( 0 , 1.00000000000000000000 )
( 1 , 1.00000000000000000000 )
( 2 , 1.00000000000000000000 )
( 3 , 1.00000000000000000000 )
( 4 , 1.00000000000000000000 )
( 5 , 1.00000000000000000000 )
( 6 , 1.00000000000000000000 )
( 7 , 1.00000000000000000000 )
( 8 , 1.00000000000000000000 )
( 9 , 1.00000000000000000000 )
( 10 , 1.00000000000000000000 )
( 11 , 1.62044026988838463232 )
( 12 , 1.73899636963452164373 )
( 13 , 1.71292754465440588696 )
( 14 , 1.71521457752224271154 )
( 15 , 1.73197612646630116906 )
( 16 , 1.74917817094069749650 )
( 17 , 1.74248162788712512850 )
( 18 , 1.75239163319457902368 )
( 19 , 1.75573448921246022536 )
( 20 , 1.76731630320612143059 )
( 21 , 1.76655481994931129464 )
( 22 , 1.77416396480162008586 )
( 23 , 1.77616264163115578079 )
( 24 , 1.78370711675526716355 )
( 25 , 1.76956314785353230334 )
( 26 , 1.76927025761818423888 )
( 27 , 1.77368453395218561575 )
( 28 , 1.76750135380246731209 )
( 29 , 1.77577981808110418882 )
( 30 , 1.76480217201386269110 )
( 31 , 1.77353673395858169037 )
( 32 , 1.77561679545887751708 )
( 33 , 1.76366264696831454017 )
( 34 , 1.76885533945150941193 )
( 35 , 1.76976499285912658443 )
( 36 , 1.77384991386121004808 )
( 37 , 1.77418481241441172713 )
( 38 , 1.77553459882849051787 )
( 39 , 1.77383643561074966222 )
( 40 , 1.77680853031814134368 )
( 41 , 1.77681234387364184180 )
( 42 , 1.77918332246715182876 )
( 43 , 1.77939972564383142029 )
( 44 , 1.78143248931554665226 )
( 45 , 1.78162699187060780126 )
( 46 , 1.78334370972043609882 )
( 47 , 1.78352802523700870729 )
( 48 , 1.78500016680562745286 )
( 49 , 1.78519006575836680639 )
( 50 , 1.78646176031758763258 )
( 51 , 1.78666473599260333671 )
( 52 , 1.78776891204103405286 )
( 53 , 1.78798545711064593711 )
( 54 , 1.78894796421145918508 )
( 55 , 1.78917424209489817066 )
( 56 , 1.79001509900157418365 )
( 57 , 1.79024648315430079393 )
( 58 , 1.79098180986885568977 )
( 59 , 1.79121365796532128201 )
( 60 , 1.79185689852412766143 )
( 61 , 1.79208509584178399265 )
( 62 , 1.79264772129803673373 )
( 63 , 1.79286890184421943140 )
( 64 , 1.79336086122853743952 )
( 65 , 1.79357245215269966998 )
( 66 , 1.79400246139144425861 )
( 67 , 1.79420264200858636805 )
( 68 , 1.79457837752711113488 )
( 69 , 1.79476599123725533147 )
( 70 , 1.79509422931387230271 )
( 71 , 1.79526867606847773118 )
( 72 , 1.79555539965864974228 )
( 73 , 1.79571652576020812653 )
( 74 , 1.79596701162781302585 )
( 75 , 1.79611500670272494773 )
( 76 , 1.79633390010668447800 )
( 77 , 1.79646920615300542678 )
( 78 , 1.79666058721700205858 )
( 79 , 1.79678382181877638324 )
( 80 , 1.79695126573131136460 )
( 81 , 1.79706315937216754008 )
( 82 , 1.79720979157498739198 )
( 83 , 1.79731113805121056615 )
( 84 , 1.79743968483399396696 )
( 85 , 1.79753130331724220170 )
( 86 , 1.79764413811035939084 )
( 87 , 1.79772684499311075079 )
( 88 , 1.79782603046597166951 )
( 89 , 1.79790061910848852911 )
( 90 , 1.79798794529355387972 )
( 91 , 1.79805517231790501498 )
( 92 , 1.79813219084189213071 )
( 93 , 1.79819276741209521653 )
( 94 , 1.79826082216521454527 )
( 95 , 1.79831540889744467826 )
( 96 , 1.79837566273843973974 )
( 97 , 1.79842486803018619668 )
( 98 , 1.79847832700610510415 )
( 99 , 1.79852270675124170651 )
( 100 , 1.79857024096580109251 )
};
 \addlegendentry{ORM};

 \addplot [ 
 color=red, 
 dashed, 
 line width=1.0pt, 
 mark size=1.8pt, 
 mark=triangle*, 
 only marks, 
  mark options={solid,fill=white,draw=red}, 
 mark repeat={2}, 
 ] 
 coordinates{ 
( 0 , 0.00000000000000000000 )
( 1 , 0.00000000000000000000 )
( 2 , 0.00000000000000000000 )
( 3 , 0.00000000000000000000 )
( 4 , 0.00000000000000000000 )
( 5 , 0.00000000000000000000 )
( 6 , 0.00000000000000000000 )
( 7 , 0.00000000000000000000 )
( 8 , 0.00000000000000000000 )
( 9 , 0.00000000000000000000 )
( 10 , 0.00000000000000000000 )
( 11 , 0.00000000000000000000 )
( 12 , 0.40972564019065171870 )
( 13 , 0.40972564019065171870 )
( 14 , 0.98019801980198129066 )
( 15 , 0.98019801980198129066 )
( 16 , 0.98019801980198129066 )
( 17 , 0.98019801980198129066 )
( 18 , 0.00000000000000000000 )
( 19 , 0.00000000000000000000 )
( 20 , 0.49724328581931720450 )
( 21 , 0.49724328581931720450 )
( 22 , 0.98019801980198129066 )
( 23 , 0.98019801980198129066 )
( 24 , 0.98019801980198129066 )
( 25 , 0.98019801980198129066 )
( 26 , 0.00000000000000000000 )
( 27 , 0.00000000000000000000 )
( 28 , 0.49268790357605030117 )
( 29 , 0.49268790357605030117 )
( 30 , 0.98019801980198129066 )
( 31 , 0.98019801980198129066 )
( 32 , 0.98019801980198129066 )
( 33 , 0.98019801980198129066 )
( 34 , 0.98019801980198129066 )
( 35 , 0.98019801980198129066 )
( 36 , 0.00000000000000000000 )
( 37 , 0.00000000000000000000 )
( 38 , 0.53866746706509860410 )
( 39 , 0.53866746706509860410 )
( 40 , 0.98019801980198129066 )
( 41 , 0.98019801980198129066 )
( 42 , 0.98019801980198129066 )
( 43 , 0.98019801980198129066 )
( 44 , 0.00000000000000000000 )
( 45 , 0.00000000000000000000 )
( 46 , 0.51223468302638930361 )
( 47 , 0.51223468302638930361 )
( 48 , 0.98019801980198129066 )
( 49 , 0.98019801980198129066 )
( 50 , 0.98019801980198129066 )
( 51 , 0.98019801980198129066 )
( 52 , 0.98019801980198129066 )
( 53 , 0.98019801980198129066 )
( 54 , 0.00000000000000000000 )
( 55 , 0.00000000000000000000 )
( 56 , 0.51887439762601139659 )
( 57 , 0.51887439762601139659 )
( 58 , 0.84875136139291729176 )
( 59 , 0.84875136139291729176 )
( 60 , 0.00000000000000000000 )
( 61 , 0.00000000000000000000 )
( 62 , 0.00000000000000000000 )
( 63 , 0.00000000000000000000 )
( 64 , 0.43626151923955086565 )
( 65 , 0.43626151923955086565 )
( 66 , 0.98019801980198129066 )
( 67 , 0.98019801980198129066 )
( 68 , 0.98019801980198129066 )
( 69 , 0.98019801980198129066 )
( 70 , 0.98019801980198129066 )
( 71 , 0.98019801980198129066 )
( 72 , 0.00000000000000000000 )
( 73 , 0.00000000000000000000 )
( 74 , 0.52363593249112871941 )
( 75 , 0.52363593249112871941 )
( 76 , 0.98019801980198129066 )
( 77 , 0.98019801980198129066 )
( 78 , 0.98019801980198129066 )
( 79 , 0.98019801980198129066 )
( 80 , 0.98019801980198129066 )
( 81 , 0.98019801980198129066 )
( 82 , 0.00000000000000000000 )
( 83 , 0.00000000000000000000 )
( 84 , 0.52481901681103160229 )
( 85 , 0.52481901681103160229 )
( 86 , 0.98019801980198129066 )
( 87 , 0.98019801980198129066 )
( 88 , 0.00000000000000000000 )
( 89 , 0.00000000000000000000 )
( 90 , 0.00000000000000000000 )
( 91 , 0.00000000000000000000 )
( 92 , 0.47324913645120958439 )
( 93 , 0.47324913645120958439 )
( 94 , 0.98019801980198129066 )
( 95 , 0.98019801980198129066 )
( 96 , 0.98019801980198129066 )
( 97 , 0.98019801980198129066 )
( 98 , 0.00000000000000000000 )
( 99 , 0.00000000000000000000 )
( 100 , 0.00000000000000000000 )
};
 \addlegendentry{OIM};

 \addplot [ 
 color=magenta, 
 dashed, 
 line width=1.0pt, 
 mark size=1.5pt, 
 mark=*, 
 only marks, 
  mark options={solid,fill=white,draw=magenta}, 
 mark repeat={2}, 
 ] 
 coordinates{ 
( 0 , 0.00000000000000000000 )
( 1 , 0.00000000000000000000 )
( 2 , 0.00000000000000000000 )
( 3 , 0.00000000000000000000 )
( 4 , 0.00000000000000000000 )
( 5 , 0.00000000000000000000 )
( 6 , 0.00000000000000000000 )
( 7 , 0.00000000000000000000 )
( 8 , 0.00000000000000000000 )
( 9 , 0.00000000000000000000 )
( 10 , 0.00000000000000000000 )
( 11 , 0.00000000000000000000 )
( 12 , 0.00000000000000000000 )
( 13 , 0.00000000000000000000 )
( 14 , 2.54556150197203123753 )
( 15 , 2.54556150197203123753 )
( 16 , 2.54556150197203123753 )
( 17 , 2.54556150197203123753 )
( 18 , 3.00142093196503356722 )
( 19 , 3.00142093196503356722 )
( 20 , 3.00142093196503356722 )
( 21 , 3.00142093196503356722 )
( 22 , 3.24673794192681341997 )
( 23 , 3.24673794192681341997 )
( 24 , 3.24673794192681341997 )
( 25 , 3.24673794192681341997 )
( 26 , 3.29233263000244713581 )
( 27 , 3.29233263000244713581 )
( 28 , 3.29233263000244713581 )
( 29 , 3.29233263000244713581 )
( 30 , 3.31509225960606546124 )
( 31 , 3.31509225960606546124 )
( 32 , 3.31509225960606546124 )
( 33 , 3.31509225960606546124 )
( 34 , 3.45547723783883320436 )
( 35 , 3.45547723783883320436 )
( 36 , 3.45547723783883320436 )
( 37 , 3.45547723783883320436 )
( 38 , 3.94651859605510679785 )
( 39 , 3.94651859605510679785 )
( 40 , 3.94651859605510679785 )
( 41 , 3.94651859605510679785 )
( 42 , 0.00000000000000000000 )
( 43 , 0.00000000000000000000 )
( 44 , 0.00000000000000000000 )
( 45 , 0.00000000000000000000 )
( 46 , 0.65657413505548045674 )
( 47 , 0.65657413505548045674 )
( 48 , 0.65657413505548045674 )
( 49 , 0.65657413505548045674 )
( 50 , 3.15565281284428200692 )
( 51 , 3.15565281284428200692 )
( 52 , 3.15565281284428200692 )
( 53 , 3.15565281284428200692 )
( 54 , 3.42220358586815764923 )
( 55 , 3.42220358586815764923 )
( 56 , 3.42220358586815764923 )
( 57 , 3.42220358586815764923 )
( 58 , 3.44868173089781215879 )
( 59 , 3.44868173089781215879 )
( 60 , 3.44868173089781215879 )
( 61 , 3.44868173089781215879 )
( 62 , 3.47666386356511436517 )
( 63 , 3.47666386356511436517 )
( 64 , 3.47666386356511436517 )
( 65 , 3.47666386356511436517 )
( 66 , 3.52036538137412957639 )
( 67 , 3.52036538137412957639 )
( 68 , 3.52036538137412957639 )
( 69 , 3.52036538137412957639 )
( 70 , 3.59988409109486262594 )
( 71 , 3.59988409109486262594 )
( 72 , 3.59988409109486262594 )
( 73 , 3.59988409109486262594 )
( 74 , 0.00000000000000000000 )
( 75 , 0.00000000000000000000 )
( 76 , 0.00000000000000000000 )
( 77 , 0.00000000000000000000 )
( 78 , 0.70842062372830216788 )
( 79 , 0.70842062372830216788 )
( 80 , 0.70842062372830216788 )
( 81 , 0.70842062372830216788 )
( 82 , 3.41855050806252114270 )
( 83 , 3.41855050806252114270 )
( 84 , 3.41855050806252114270 )
( 85 , 3.41855050806252114270 )
( 86 , 3.48717448038827626533 )
( 87 , 3.48717448038827626533 )
( 88 , 3.48717448038827626533 )
( 89 , 3.48717448038827626533 )
( 90 , 3.49038596109165144554 )
( 91 , 3.49038596109165144554 )
( 92 , 3.49038596109165144554 )
( 93 , 3.49038596109165144554 )
( 94 , 3.49323609684675195908 )
( 95 , 3.49323609684675195908 )
( 96 , 3.49323609684675195908 )
( 97 , 3.49323609684675195908 )
( 98 , 3.50358496025908072724 )
( 99 , 3.50358496025908072724 )
( 100 , 3.50358496025908072724 )
};
 \addlegendentry{OAIM};

  \end{axis} 
 \end{tikzpicture}

%% file: acceleration.bbl
\def\cprime{$'$} \def\cdprime{$''$} \def\cprime{$'$} \def\cprime{$'$}
\begin{thebibliography}{10}

\bibitem{alvarez2004weak}
{\sc F.~Alvarez}, {\em Weak convergence of a relaxed and inertial hybrid
  projection-proximal point algorithm for maximal monotone operators in hilbert
  space}, SIAM Journal on Optimization, 14 (2004), pp.~773--782.

\bibitem{alvarez2001inertial}
{\sc F.~Alvarez and H.~Attouch}, {\em An inertial proximal method for maximal
  monotone operators via discretization of a nonlinear oscillator with
  damping}, Set-Valued Analysis, 9 (2001), pp.~3--11.

\bibitem{attouch2015rate}
{\sc H.~Attouch and J.~Peypouquet}, {\em The rate of convergence of nesterov's
  accelerated forward-backward method is actually $ o(k^{-2})$}, arXiv preprint
  arXiv:1510.08740,  (2015).

\bibitem{bauschke2014optimal}
{\sc H.~Bauschke, J.~Bello~Cruz, T.~Nghia, H.~Phan, and X.~Wang}, {\em Optimal
  rates of linear convergence of relaxed alternating projections and
  generalized douglas-rachford methods for two subspaces.}, Numerical
  Algorithms,  (in press), pp.~1--44.

\bibitem{livre-combettes}
{\sc H.~H. Bauschke and P.~L. Combettes}, {\em Convex analysis and monotone
  operator theory in {H}ilbert spaces}, CMS Books in Mathematics/Ouvrages de
  Math\'ematiques de la SMC, Springer, New York, 2011.

\bibitem{beck2009fast}
{\sc A.~Beck and M.~Teboulle}, {\em A fast iterative shrinkage-thresholding
  algorithm for linear inverse problems}, SIAM journal on imaging sciences, 2
  (2009), pp.~183--202.

\bibitem{bianchi2015stochastic}
{\sc P.~Bianchi, W.~Hachem, and F.~Iutzeler}, {\em A coordinate descent
  primal-dual algorithm and application to distributed asynchronous
  optimization}, arXiv preprint arXiv:1407.0898,  (2014).

\bibitem{boyd2011distributed}
{\sc S.~Boyd, N.~Parikh, E.~Chu, B.~Peleato, and J.~Eckstein}, {\em Distributed
  optimization and statistical learning via the alternating direction method of
  multipliers}, Foundations and Trends in Machine Learning, 3 (2011),
  pp.~1--122.

\bibitem{chambolle2014convergence}
{\sc A.~Chambolle and C.~Dossal}, {\em On the convergence of the iterates of
  ``fista''.}, Preprint hal-01060130, September,  (2014).

\bibitem{condat2013primal}
{\sc L.~Condat}, {\em A primal--dual splitting method for convex optimization
  involving lipschitzian, proximable and linear composite terms}, Journal of
  Optimization Theory and Applications, 158 (2013), pp.~460--479.

\bibitem{eckstein1992douglas}
{\sc J.~Eckstein and D.~P. Bertsekas}, {\em On the {D}ouglas-{R}achford
  splitting method and the proximal point algorithm for maximal monotone
  operators}, Mathematical Programming, 55 (1992), pp.~293--318.

\bibitem{flammarion2015averaging}
{\sc N.~Flammarion and F.~Bach}, {\em From averaging to acceleration, there is
  only a step-size}, arXiv preprint arXiv:1504.01577,  (2015).

\bibitem{ghadimi2013}
{\sc E.~Ghadimi, A.~Teixeira, I.~Shames, and M.~Johansson}, {\em Optimal
  parameter selection for the alternating direction method of multipliers
  (admm): Quadratic problems}, arXiv preprint arXiv:1306.2454,  (2013).

\bibitem{giselsson2016line}
{\sc P.~Giselsson, M.~F{\"a}lt, and S.~Boyd}, {\em Line search for averaged
  operator iteration}, arXiv preprint arXiv:1603.06772,  (2016).

\bibitem{goldstein2014fast}
{\sc T.~Goldstein, B.~O'Donoghue, S.~Setzer, and R.~Baraniuk}, {\em Fast
  alternating direction optimization methods}, SIAM Journal on Imaging
  Sciences, 7 (2014), pp.~1588--1623.

\bibitem{HorJoh94}
{\sc R.~A. Horn and C.~R. Johnson}, {\em Matrix analysis}, Cambridge University
  Press, 2007.

\bibitem{iut-cdc13}
{\sc F.~Iutzeler, P.~Bianchi, P.~Ciblat, and W.~Hachem}, {\em Asynchronous
  distributed optimization using a randomized {A}lternating {D}irection
  {M}ethod of {M}ultipliers}, in Proc. IEEE Conf. Decision and Control (CDC),
  Florence, Italy, Dec. 2013.

\bibitem{iutzeler2013explicit}
{\sc F.~Iutzeler, P.~Bianchi, P.~Ciblat, and W.~Hachem}, {\em Explicit
  convergence rate of a distributed alternating direction method of
  multipliers}, arXiv preprint arXiv:1312.1085,  (2013).

\bibitem{iutzeler2013analysis}
{\sc F.~Iutzeler, P.~Ciblat, and W.~Hachem}, {\em Analysis of sum-weight-like
  algorithms for averaging in wireless sensor networks}, IEEE Transactions on
  Signal Processing, 61 (2013), pp.~2802--2814.

\bibitem{johnstone2015lyapunov}
{\sc P.~Johnstone and P.~Moulin}, {\em A lyapunov analysis of fista with local
  linear convergence for sparse optimization}, arXiv preprint arXiv:1502.02281,
   (2015).

\bibitem{2015arXiv150602186L}
{\sc H.~{Lin}, J.~{Mairal}, and Z.~{Harchaoui}}, {\em {A Universal Catalyst for
  First-Order Optimization}}, ArXiv e-prints,  (2015),
  \href{http://arxiv.org/abs/1506.02186}{arXiv:1506.02186}.

\bibitem{lin2014adaptive}
{\sc Q.~Lin and L.~Xiao}, {\em An adaptive accelerated proximal gradient method
  and its homotopy continuation for sparse optimization}, Computational
  Optimization and Applications, 60 (2014), pp.~633--674.

\bibitem{lions1979splitting}
{\sc P.-L. Lions and B.~Mercier}, {\em Splitting algorithms for the sum of two
  nonlinear operators}, SIAM Journal on Numerical Analysis, 16 (1979),
  pp.~964--979.

\bibitem{lorenz2014inertial}
{\sc D.~Lorenz and T.~Pock}, {\em An inertial forward-backward algorithm for
  monotone inclusions}, Journal of Mathematical Imaging and Vision, 51 (2014),
  pp.~311--325.

\bibitem{mainge2008convergence}
{\sc P.-E. Maing{\'e}}, {\em Convergence theorems for inertial km-type
  algorithms}, Journal of Computational and Applied Mathematics, 219 (2008),
  pp.~223--236.

\bibitem{mu2015note}
{\sc Z.~Mu and Y.~Peng}, {\em A note on the inertial proximal point method},
  Statistics, Optimization \& Information Computing, 3 (2015), pp.~241--248.

\bibitem{nesterov1983method}
{\sc Y.~Nesterov}, {\em A method of solving a convex programming problem with
  convergence rate o (1/k2)}, Soviet Mathematics Doklady, 27 (1983),
  pp.~372--376.

\bibitem{nesterov2005smooth}
{\sc Y.~Nesterov}, {\em Smooth minimization of non-smooth functions},
  Mathematical programming, 103 (2005), pp.~127--152.

\bibitem{o2013adaptive}
{\sc B.~O'Donoghue and E.~Candes}, {\em Adaptive restart for accelerated
  gradient schemes}, Foundations of computational mathematics, 15 (2013),
  pp.~715--732.

\bibitem{polyak1964some}
{\sc B.~Polyak}, {\em Some methods of speeding up the convergence of iteration
  methods}, USSR Computational Mathematics and Mathematical Physics, 4 (1964),
  pp.~1--17.

\bibitem{richardson1911approximate}
{\sc L.~F. Richardson}, {\em The approximate arithmetical solution by finite
  differences of physical problems involving differential equations, with an
  application to the stresses in a masonry dam}, Philosophical Transactions of
  the Royal Society of London,  (1911), pp.~307--357.

\bibitem{saad2003iterative}
{\sc Y.~Saad}, {\em Iterative methods for sparse linear systems}, Siam, 2003.

\bibitem{shi-etal-(arxiv)13}
{\sc W.~{Shi}, Q.~{Ling}, K.~{Yuan}, G.~{Wu}, and W.~{Yin}}, {\em {On the
  Linear Convergence of the ADMM in Decentralized Consensus Optimization}},
  ArXiv e-prints,  (2013),
  \href{http://arxiv.org/abs/1307.5561}{arXiv:1307.5561}.

\bibitem{shiu1976cyclically}
{\sc E.~Shiu}, {\em Cyclically monotone linear operators}, Proceedings of the
  American Mathematical Society, 59 (1976), pp.~127--132.

\bibitem{tao2015local}
{\sc S.~Tao, D.~Boley, and S.~Zhang}, {\em Local linear convergence of ista and
  fista on the lasso problem}, arXiv preprint arXiv:1501.02888,  (2015).

\bibitem{taylor2015smooth}
{\sc A.~Taylor, J.~Hendrickx, and F.~Glineur}, {\em Smooth strongly convex
  interpolation and exact worst-case performance of first-order methods}, arXiv
  preprint arXiv:1502.05666,  (2015).

\bibitem{tseng2008accelerated}
{\sc P.~Tseng}, {\em On accelerated proximal gradient methods for
  convex-concave optimization}, submitted to SIAM Journal on Optimization,
  (2008).

\bibitem{wei-ozd-arxiv13}
{\sc E.~Wei and A.~Ozdaglar}, {\em On the {$O(1/k)$} convergence of
  asynchronous distributed {A}lternating {D}irection {M}ethod of
  {M}ultipliers}, arXiv preprint arXiv:1307.8254,  (2013).

\end{thebibliography}
